\documentclass[reqno,draft]{amsart}
\usepackage{amsmath,amssymb,bbm}
\newtheorem{theorem}{Theorem}[section]
\newtheorem{lemma}[theorem]{Lemma}
\newtheorem{corollary}[theorem]{Corollary}

\theoremstyle{remark}
\newtheorem{remark}[theorem]{Remark} 
 
\newtheorem{example}[theorem]{Example}

\numberwithin{equation}{section}

\newcommand {\rn}{{\mathrm n}}
\newcommand {\rb}{{\mathrm b}}

\newcommand{\kX}{{\mathcal X}}
\newcommand{\kE}{{\mathcal E}}
\newcommand{\kY}{{\mathcal Y}}
\newcommand{\kB}{{\mathcal B}}
\newcommand{\kC}{{\mathcal C}}
\newcommand{\kV}{{\mathcal V}}
\newcommand{\kW}{{\mathcal W}}
\newcommand{\kD}{{\mathcal D}}
\newcommand{\kL}{{\mathcal L}}
\newcommand{\kH}{{\mathcal H}}
\newcommand{\kU}{{\mathcal U}}

\newcommand{\kI}{{\mathcal I}}

\newcommand{\kZ}{{\mathcal Z}}

\newcommand{\kCb}{\kC_{\operatorname{b}}}

\newcommand  {\D}{{\mathrm D}}
\renewcommand{\d}{{\mathrm d}}
\renewcommand{\i}{{\mathrm i}}
\newcommand  {\e}{{\mathrm e}}

\renewcommand{\a}{{\mathsf a}}
\newcommand {\mS}{{\mathsf S}}

\renewcommand{\b}{{\mathsf b}}

\newcommand  {\1}{\mathbbm{1}}
\newcommand  {\R}{{\mathbb R}}
\newcommand  {\N}{{\mathbb N}}
\newcommand  {\C}{{\mathbb C}}
   
\newcommand  {\Z}{{\mathbb Z}}
\renewcommand{\P}{{\mathbb P}}
\newcommand  {\Q}{{\mathbb Q}}

\newcommand {\id}{\operatorname{id}}
\newcommand {\spec}{\operatorname{spec}}

\renewcommand {\Re}{\operatorname{Re}}

\newcommand{\dist}{\operatorname{dist}}
\newcommand{\interior}{\operatorname{int}}

\newcommand{\norm}[3][]{%
  \mathchoice{\lVert #2 \rVert_{#3}^{#1\vphantom\int}}%
             {\lVert #2 \rVert_{#3}^{#1}}{}{}}
\newcommand{\tnorm}[2]{%
  \lVert #1 \rVert_{#2}}
\def\vvvert{\hbox{\ensuremath{|\hspace{-0.16em}|\hspace{-0.16em}|}}}
\newcommand{\Norm}[2]{%
  \mathchoice{\vvvert #1 \vvvert\vphantom\vert_{#2}^{\vphantom\int}}%
             {\vvvert #1 \vvvert\vphantom\vert_{#2}}{}{}}

\begin{document}
\title[Galerkin approximations of A-stable Runge--Kutta discretizations]
{Stability under Galerkin truncation of
A-stable Runge--Kutta discretizations in time}

\author[M. Oliver]{Marcel Oliver}
\address[M. Oliver]%
{School of Engineering and Science \\
 Jacobs University \\
 28759 Bremen \\
 Germany}
\email{oliver@member.ams.org}

\author[C. Wulff]{Claudia Wulff}
\address[C. Wulff]%
{Department of Mathematics\\
 University of Surrey \\
 Guildford GU2 7XH \\
 UK}
\email{c.wulff@surrey.ac.uk}

\date{\today}
 
\begin{abstract}
We consider semilinear evolution equations for which the linear part
is normal and generates a strongly continuous semigroup and the
nonlinear part is sufficiently smooth on a scale of Hilbert spaces.
We approximate their semiflow by an implicit, A-stable Runge--Kutta
discretization in time and a spectral Galerkin truncation in space.
We show regularity of the Galerkin-truncated semiflow and its
time-discretization on open sets of initial values with bounds that
are uniform in the spatial resolution and the initial value.  We also
prove convergence of the space-time discretization without any
condition that couples the time step to the spatial resolution.  Then
we estimate the Galerkin truncation error for the semiflow of the
evolution equation, its Runge--Kutta discretization, and their
respective derivatives, showing how the order of the Galerkin
truncation error depends on the smoothness of the initial data.  Our
results apply, in particular, to the semilinear wave equation and to
the nonlinear Schr\"odinger equation.
\end{abstract}

\maketitle

\tableofcontents

\section{Introduction}

We study semilinear evolution equations
\begin{equation}
  \label{e.pde}
  \partial_t U  = F(U)  = AU + B(U) 
\end{equation}
posed on a Hilbert space $\kY$ under spectral spatial Galerkin
truncation and temporal discretization by a large class of A-stable
Runge--Kutta methods.  The methods considered permit a well-defined
temporal semi-discretization \cite{LubOst93,OW2010}; particular
examples are Gauss--Legendre Runge--Kutta schemes.  The linear
operator $A$ of \eqref{e.pde} is assumed to be normal and to generate
a strongly continuous, not necessary analytic semigroup; $B$ is a
bounded nonlinear operator on $\kY$.  (This setting includes, without
loss of generality, cases where $A$ is normal up to a bounded
perturbation as a bounded non-normal part can always be included into
the operator $B$).  The examples we have in mind are semilinear
Hamiltonian evolution equations such as the semilinear wave equation
or the nonlinear Schr\"odinger equation, though for the results in
this paper we do not assume a Hamiltonian structure.

Differentiation of the semiflow in time results in multiplication with
the unbounded operator $A$.  Hence, in general, the time derivative of
the semiflow is only well-defined when considered as a map from a
subset of $D(A)$ to $\kY$ \cite{Pazy}; to be able to differentiate
repeatedly, we assume that $B$ is $\kC^{N-k}$ as a map from some open
set $\kD_k \subset \kY_k \equiv D(A^k)$ to $\kY_k$ for $k=0, \ldots,
K$ and $N>K$.  This is formalized as condition (B1) in the main text
of the paper.  We prove that the semiflow of the Galerkin truncated
evolution equation and its temporal discretization are of class
$\kC^K$ jointly in time (resp.\ stepsize) and in the initial data when
considered as a map from $\kD_K \subset \kY_K$ to $\kY$, with uniform
bounds in the spatial resolution.  Analogous results hold true for the
full semiflow and its time-semidiscretization \cite{OW2010}.  We prove
full-order convergence of the space-time discretization on open sets
of initial data without the need of a Courant condition that couples
spatial and temporal resolution.  We then provide estimates on the
truncation error of the Galerkin approximation of the semiflow, the
temporal discretization and their derivatives, and study the
dependence of the order of the truncation error on the smoothness of
the initial data.

When implicit Runge--Kutta methods are applied to stiff problems, they
often converge at less than their formal order of convergence.  This
phenomenon is called \emph{order reduction} \cite{BW96}.  For
time-semidiscretizations of initial-boundary value problems,
order-reduction can be tied to lack of regularity \cite{BrennerThomee}
or mismatch of boundary conditions in the internal stages of the
method \cite{AM02}; both papers give conditions under which full-order
convergence is achieved.  In our work, we are in the setting of
\cite[Theorem 3]{BrennerThomee} except that we consider semilinear
equations.  For linear evolution equations, this earlier result gives
order $p$ convergence when $p$ is the formal order of the method and
the initial data is in $D(A^{p+1})$.  When considering semilinear
problems, our condition (B1) on the mapping properties of the
nonlinearity typically imposes additional boundary conditions that the
nonlinearity $B$ has to match if the operator $A$ has boundary
conditions which are not periodic.  Condition (B1), together with the
assumption that the initial data lie in $D(A^{p+1})$, enforces
matching boundary conditions and excludes order reduction.

The standard requirement for the existence of a semiflow is Lipshitz
continuity of the nonlinearity $B$ \cite{Pazy}.  It holds true for a
large class of evolution equations and will be referred to as
condition (B0) in the main text of the paper.  Our assumption (B1)
implies Lipshitz continuity.  Whether the stronger condition (B1)
holds true with $K>0$ depends nontrivially on the evolution equation
and its boundary conditions.  It is satisfied by our main examples,
the semilinear wave equation and nonlinear Schr\"odinger equation with
smooth nonlinearities and periodic or homogeneous Neumann boundary
conditions.  It is also satisfied for homogeneous Dirichlet conditions
under additional conditions on the nonlinearity, see
Section~\ref{s.pde} below.  If (B1) is not satisfied for sufficiently
large $K$, we cannot ensure that the solution $U(t)$ of \eqref{e.pde}
and its numerical approximations have enough temporal smoothness to
obtain full order convergence of the time-discretization independent
of the spatial resolution.

We recall from \cite{Pazy} that the solution to the full semilinear
evolution equation is obtained as a fixed point of a contraction map,
which we consider on the scale of Hilbert spaces $\kY_0, \dots,
\kY_K$.  Similarly, the Runge--Kutta temporal discretizations are
functions of the Runge--Kutta stage vectors, which in turn are
obtained as fixed points of contraction maps.  Remaining in this
setting, we now consider spatial Galerkin approximation as a
perturbation of these contraction maps.  To do so, we provide an
abstract theory for the stability of fixed points under perturbation
of contraction mappings on a scale of Banach spaces, thereby extending
the theory of contraction maps on scales of Banach spaces from
\cite{OW2010,Vand87,Wulff00}.  This theory provides us with a unified
framework for the time-continuous and the time-discrete cases.

Let us mention some related results.  Spatial spectral Galerkin
approximation (also called Faedo--Galerkin approximation) is
frequently used as a theoretical tool for the construction of
solutions to partial differential equations; see, e.g.,
\cite{DoeringGibbon,Temam}.  Error estimates for smooth solutions of
parabolic problems under spectral and more general Galerkin
approximations (such as finite element methods) can be found, e.g., in
\cite{ThomeeEtAl,ThomeeGalerkin}.  In the parabolic case, there has
been a lot of interest in the so-called nonlinear Galerkin method
which has been shown to have a better convergence rate than the
standard spectral Galerkin method, see \cite{Titi,MTW03} and
references therein.  For analytic initial data, an exponential rate of
convergence of the Galerkin approximation to the semiflow of the
Ginzburg--Landau equation has been shown in \cite{Ginzburg}.

Hyperbolic problems, namely the semilinear wave equation, and their
discretizations have been studied, e.g., by Baker \emph{et al.}\
\cite{BakerGalerkinHyp}.  They provide estimates for the order of
convergence of the spatial Galerkin method of the semiflow for smooth
enough data and globally Lipshitz nonlinearities under an assumption
on the elliptic projection of the solution; they also treat explicit
multistep time discretizations of the spatial approximation under a
Courant condition that couples the accuracy of the Galerkin method
with the time-stepsize.  Bazley \cite{Bazley} shows the convergence of
the Faedo--Galerkin approximations of the semilinear wave equation for
a special class of nonlinearities on the interval of existence of the
continuous solution.  Verver and Sanz-Serna \cite{VS84} identify
general consistency and stability conditions in which convergence of
spatial semidiscretizations and of their temporal discretizations can
be proved.  They further verify these conditions for a nonlinear
parabolic PDE and for the cubic nonlinear Schr\"odinger equation.  In
this paper we provide a general framework in which those conditions
hold true with uniform bounds on open sets of initial data.  Miklavcic
\cite{Miklavcic} studies a class of parabolic and hyperbolic
semilinear evolution equations with a linear part that generates a
$\kC^0$ semigroup, and shows pointwise convergence of the spatial
Galerkin approximations of the semiflow; he considers nonlinearities
$B$ which are Lipshitz on the whole of $\kY$.  Karakashian \emph{et
al.}\ \cite{RK-NLS-98} study a class of implicit Runge--Kutta
time-discretizations (including Gauss--Legendre methods) and spatial
Galerkin approximations for the cubic nonlinear Schr\"odinger
equations and prove convergence for smooth solutions under mesh
conditions that couple spatial and temporal resolution.
  
In this paper, the emphasis is on estimates for the spatial Galerkin
truncation error of the joint higher order derivatives in time and in
the initial data both of the semiflow and of its temporal
discretization.  In contrast to \cite{BakerGalerkinHyp,RK-NLS-98}, our
estimates for the numerical method hold uniformly in the time-stepsize
and do not require conditions that couple the spatial and temporal
accuracy of the discretization.  Our results include statements on the
pointwise convergence of Galerkin spatial semi-discretizations for
non-smooth solutions of \eqref{e.pde} on their interval of existence,
see Theorem~\ref{t.wpp}.  These are similar to the results of
\cite{Bazley,Miklavcic}, but include more general evolution equations.
Our results yield algebraic orders of the Galerkin truncation error
for smooth, but non-analytic initial data.  However, using the methods
developed here, it is also possible to obtain exponential estimates
for analytic data as in \cite{Ginzburg}.
 
The paper is organized as follows.  In Section~\ref{s.pde}, we
introduce the class of semilinear evolution equations considered, and
show how the semilinear wave equation and the nonlinear Schr\"odinger
equation fit into this framework for different types of boundary
conditions.  In this setting, we study regularity and stability under
Galerkin truncation of the semiflow.  In Section~\ref{s.rk}, we apply
a class of A-stable Runge--Kutta methods to the semiflow of the
Galerkin truncated evolution equation and prove results on regularity
and stability under Galerkin truncation for the temporal
discretization which are analogous to the semiflow.  We also study
convergence of the space-time discretization.

We present our results in two versions: we label results that provide
uniformity of the time interval of existence (for the semiflow) and
the maximum time step (for the numerical method) on sufficiently small
balls of initial data as ``local version.''  Assuming more regularity
for the initial data, we also obtain results which are uniform on
bounded open sets so long as $B$ is well-defined and bounded.  We will
label results of this type by ``uniform version.''

In the appendix, we present a number of technical results on stability
of fixed points of contraction maps on scales of Banach spaces, which
are needed in the main body of the paper.


\section{Semilinear evolution equations under Galerkin truncation}
\label{s.pde}
 
We begin by introducing the class of semilinear evolution equations
which we study in this paper.  We then prove regularity of the
Galerkin truncated semiflow with uniform bounds in its spatial
resolution and analyze the dependency of the truncation error of the
semiflow and its derivatives on the smoothness of the initial data.


\subsection{General setting}
\label{ss.genSett}

We consider the semilinear evolution equation \eqref{e.pde} on a
Hilbert space $\kY$ and assume the following.

\begin{itemize}
\item[(A)] $A$ is a normal operator on a Hilbert space $\kY$ which
generates a $\kC^0$-semigroup $\e^{tA}$.

\item[(B0)] $B \colon \kD \to \kY$ is Lipshitz.
\end{itemize}
Recall that an operator $A$ is normal if it is closed and $AA^* = A^*
A$.  For a definition of strongly continuous semigroups
($\kC^0$-semigroups), see \cite{Pazy}.  Assumption (A) implies that
there exists $\omega \in \R$ such that
\begin{equation}
  \label{e.etAEstimate}
  \Re (\spec A) \leq \omega 
  \qquad \mbox{and} \qquad 
  \lVert \e^{tA} \rVert \leq  \, \e^{\omega t}
\end{equation}
for all $t \geq 0$.  In case $A= A_\rn + A_\rb$ where $A_\rn$
satisfies (A) and $A_\rb$ is bounded, we can redefine $B$ as $B+A_\rb$
and and $A$ as $A_\rn$, whence conditions (A) and (B0) hold true.
This situation is typical for semilinear wave equations, see Example
\ref{ex.swe}.
 
For fixed $T>0$ and $U^0 \in \kD$ let $W \in \kC([0,1];\kY)$ satisfy
the fixed point equation $W = \Pi(W;U^0,T)$ where, for $\tau \in
[0,1]$,
\begin{equation}
  \label{e.PiFlow}
  \Pi(W;U^0,T)(\tau) = \e^{\tau T A} U^0 +
  T \int_0^\tau \e^{(\tau-\sigma)TA} \, B(W(\sigma)) \, \d\sigma \,.
\end{equation}
When $T$ is small enough, $\Pi$ is a contraction on the space
$\kCb([0,1];\kY)$ so that the contraction mapping theorem implies the
existence of a fixed point \cite{Pazy}.  We then define the semiflow
$\Phi$ of \eqref{e.pde} by $\Phi^{\tau T}(U^0) = W(U^0,T)(\tau)$.  We
sometimes write $\Phi^t$ to denote the map $\Phi(\,\cdot\,,t)$.

It is apparent from \eqref{e.PiFlow} with $B=0$ that the $\ell$-th
time derivative of $U(t)$ is in $\kY$ only if $U^0 \in D(A^\ell)$.
Hence, we work on a hierarchy of Hilbert spaces defined as follows.
We set $\kY_0 \equiv \kY$ and, for $\ell \in \N$, we define $\kY_\ell
\equiv D(A^\ell)$ endowed with scalar product
\[
  \langle U_1, U_2 \rangle_{\kY_{ \ell}}
   = \langle \P U_1, \P U_2 \rangle_{\kY}   + 
     \langle 
       \lvert  A \rvert^{\ell} \, \Q U_1,
       \lvert   A \rvert^{\ell} \, \Q U_2
     \rangle_{\kY} \,.
\]
Here $\P \equiv \P_1$ is the spectral projector of $A$ onto the set
$\{\lambda \in \spec (A) \colon \lvert \lambda \rvert \leq 1 \}$, and
$\Q = 1 - \P$.  This definition of the norm ensures that
\begin{equation}
  \label{e.normA}
  \norm{A}{\kY_{\ell+1} \to \kY_{\ell}} \leq  1
  \quad \text{and} \quad
  \norm{U}{\kY_{\ell}} \leq  \norm{U}{\kY_{\ell+1}}
\end{equation}
for all $U \in \kY_{\ell+1}$.  

Let $\kD \subset \kY$ be open.  We define
\begin{equation}
  \label{e.kD-delta}
  \kD^{-\delta}  
  = \{ U \in \kD  \colon 
       \dist_{\kY } (U, \partial \kD ) > \delta \} \,.
\end{equation}
Given $\delta>0$ and a hierarchy of open sets $\kD_\ell \subset
\kY_\ell$ for $\ell = 0, \ldots, L$ for $L \in \N$ with $\kD_0 \equiv
\kD$, we define $\kD_0^{-\delta} \equiv \kD^{-\delta}$ as in
\eqref{e.kD-delta} and, for $\ell=1, \dots, L$,
\begin{equation}
  \label{e.Dk-delta}
  \kD^{-\delta}_\ell 
  \equiv \{ U \in \kD_\ell \colon 
           \dist_{\kY_\ell} (U, \partial \kD_\ell) > \delta \}  \,.
\end{equation}
Then, by construction, $\kB_{\delta}^{\kY_\ell}(U) \subset \kD_\ell$
for all $U \in \kD^{-\delta}_\ell$ and $\ell=0, \ldots, L$ where, for
any Banach space $\kX$ and $X^0 \in \kX$, we write
\[
  \kB_R^\kX(X^0) = \{ X \in \kX \colon \norm{X-X^0}\kX \leq R \}
\]
to denote the closed ball of radius $R$ around $X^0$.

Let $\kY_1$ be a Banach space continuously embedded into the Banach
space $\kY$.  Then $\kD_1 \subset \kY_1$ is called a $\delta_*$-nested
subset of $\kD \subset \kY$ if $\kD^{-\delta}_1 \subset \kD^{-\delta}
$ for all $\delta \in [0,\delta_*]$.  Furthermore we say that the
family $\kD_0, \dots, \kD_L$ is $\delta_*$-nested if
\[
\kD^{-\delta}_\ell \subset \kD^{-\delta}_{\ell-1} ~~\mbox{for all}~~\delta
\in [0,\delta_*] \mbox{ and}~~ \ell = 1,\ldots, L,
\]  
 with $\delta_*>0$.
For example, the family $\kD_k = \kB_R^{\kY_k}(U^0)$ is $\delta_*$-nested
for every $\delta_* \in (0,R)$ and $U^0 \in \kY_L$.  However, an
arbitrary nested family $\kD_\ell \subset \kY_\ell$ may not be
$\delta_*$-nested for any $\delta_*>0$.

Finally, we write $B \in \kCb^{N}(\kD, \kY)$ for some $\kD \subset
\kY$ if $B \in \kC^{N}(\kD, \kY)$ and if its derivatives are, in
addition, bounded and extend continuously to the boundary.  Then we
can state a condition under which $\Phi$ defines a semiflow on the
scale of spaces $\kY_0, \dots, \kY_K$.

\begin{itemize}
\item[(B1)] There exist $K \in \N_0$, $N \in \N$ with $N>K$, and a
$\delta_*$-nested sequence of $\kY_k$-bounded and open sets $\kD_k
\subset \kY_k$ such that $B \in \kCb^{N-k}(\kD_k, \kY_k)$ for $k = 0,
\dots, K$.
\end{itemize}
We denote the bounds of the maps $B \colon \kD_k \to \kY_k$ and their
derivatives by constants $M_k$, $M_k'$, etc., for $k = 0, \dots, K$
and set $M=M_0$, $M'=M'_0$, and so forth.  In addition to the domains
$\kD_0, \dots, \kD_K$ defined in this assumption, we will sometimes
need to refer to $\kD_{K+1}$, which may be any $\delta_*$-nested
subset of $\kD_K$ which is bounded and open in $\kY_{K+1}$.

We now give two examples of PDEs that satisfy assumptions (A) and
(B1).

\begin{example}[Functional setting for the semilinear wave equation]
\label{ex.swe} 
For the semilinear wave equation 
\begin{equation}
  \label{e.swe}
  \partial_{tt}u = \partial_{xx} u - f(u)
\end{equation} 
on $I=(0,1)$ with periodic boundary conditions $u(0)=u(1)$, we set
$U=(u,v)$ and
\[
  \kY_{\ell} = \kH_{\ell+1}(I;\R) \times \kH_{\ell}(I;\R)
\]
for $\ell \in \N$.  Here, $\kH_{\ell}(I;\R)$ denotes the Sobolev space
of square integrable functions whose first $\ell$ weak derivatives are
square-integrable.  Then the operators $A$ and $B$ are given by
\begin{equation}
  \label{e.AB}
 \tilde A = 
  \begin{pmatrix}
  0 & \id \\ \partial^2_x & 0
  \end{pmatrix} \,, \quad 
  A = (1-\P_0) \tilde A \,,
  \quad \text{and} \quad
  B(U) =   
  \begin{pmatrix}
    u \\ -f(u)
  \end{pmatrix} \,,
\end{equation}
where $\P_0$ is the spectral projector of $\tilde A$ to the eigenvalue
$0$.  Note that we have moved $\P_0 \tilde A U$ into the nonlinearity
$B$ as $\P_0 \tilde A$ is not normal.  Then the group generated by $A$
is unitary on any $\kY_{\ell}$ and $A$ generates a $\kC^0$-group on
$\kY_{\ell}$.  So, assumption (A) is satisfied.  If the nonlinearity
$f$ of the semilinear wave equation \eqref{e.swe} is, e.g., a
polynomial, then (B1) is satisfied for any $K$ and $N$ as $\kH_\ell$
is a topological algebra for $\ell>1/2$ \cite{Adams}.  More generally,
if $f \in \kC^N(D,\R)$ for some $N\in \N$ and $D\subset \R$ open, then
(B1) holds for $K<N$; see, e.g., \cite[Theorem 2.12]{OW2010}.  The
same holds true in the case of homogeneous Neumann boundary
conditions.  For homogeneous Dirichlet conditions we must additionally
require that $f^{(2j)}(0)=0$ for $0\leq 2j\leq K-1$; the same
restriction on the nonlinearity applies when $A$ is a nonconstant
coefficient operator and $K\leq 4$, see \cite[Section~2.5]{OW2010}.
\end{example}
 

\begin{example}[Functional setting for the nonlinear Schr\"odinger
equation] \label{ex.nse}
For the nonlinear Schr\"odinger equation
\begin{equation}
  \label{e.nse}
  \i \, \partial_t u 
  = - \partial_{xx} u + \partial_{\overline{u}} V(u,\overline{u})
\end{equation}
with periodic boundary conditions on $I=(0,1)$, we set $U \equiv u$
and identify
\begin{equation}
  \label{e.nlsdefs}
  A = \i \, \partial^2_x 
  \quad \text{and} \quad
  B(U) = -\i \, \partial_{\overline{u}} V(u,\overline{u}) \,. 
\end{equation}
The Laplacian is diagonal in the Fourier representation with
eigenvalues $-k^2$ where $k \in \Z$.  Hence $A$ generates a unitary
group on the square integrable functions $\kL_2 \equiv \kL_2(I;\C)$
and, more generally, on every $\kH_{\ell}(I;\C)$ with $\ell \in \N_0$.
So the operator $A$ is normal, and assumption (A) holds trivially.  In
the notation of the abstract functional setting of
Section~\ref{ss.genSett}, we choose $\kY_{\ell} =
\kH_{2\ell+1}(I;\C)$.  If the potential $V(u, \overline u)$ satisfies
$V \in \kC^{K+2+N}(D,\R^2)$ for some open subset $D \subset \R^2
\equiv \C$, then, by \cite[Theorem 2.12]{OW2010}, the nonlinearity $B$
defined in \eqref{e.nlsdefs} satisfies assumption (B1) for $K<N$ and,
in particular, (B0).
\end{example}


\subsection{Spectral Galerkin truncation and convergence}
\label{s.truncation}

We now truncate the evolution equation \eqref{e.pde} to an
$A$-invariant subspace (Galerkin subspace) as follows.  For $m \in \N$
let $\P_m$ be the sequence of spectral projectors of $A$ onto the set
$\{ \lambda \in \spec (A) \colon \lvert \lambda \rvert \leq m \}$.
Then, assumption (A) implies that
\[
  \lim_{m \to \infty} \P_m U = U
\]
for all $U \in \kY$, and
\begin{equation}
  \label{eq:PmAEstimate}
  \norm{A \P_m U}{\kY} \leq m \, \norm{\P_m U}{\kY} 
\end{equation}
for $m \in \N$.  Functions $U \in \kY_{ \ell }$ are well approximated
by their Galerkin projections $\P_m U$.  Indeed, setting $\Q_m = \id
-\P_m$,
\begin{equation}
  \norm{\Q_m U}\kY 
  \leq m^{-\ell} \, \norm{U}{\kY_{ \ell}} \,.
 \label{e.qm-estimate}
\end{equation}
We now introduce the restricted evolution equation
\begin{equation} 
  \label{e.pde-m}
\begin{aligned}
  \dot u_m 
  & =   A u_m + B_m(u_m) 
  = \P_m F(u_m) \\
  & \equiv f_m(u_m)  
  = A u_m + B_m(u_m) \,,
\end{aligned} 
\end{equation}
where $B_m = \P_m B$.  We write $\phi^t_m(\cdot)$ to denote the
semiflow of \eqref{e.pde-m} on $\P_m \kY$ and define $\Phi_m= \phi_m
\circ \P_m$.

The following theorem provides well-posedness for the projected system
on the same interval of time on which a solution to the full equation
exists, and convergence of solutions.

\begin{theorem}[Convergence of the projected system]
\label{t.wpp}
Under assumptions \textup{(A)} and \textup{(B0)}, let $U \in
\kC([0,T], \kD)$ be a mild solution to the semilinear evolution
equation \eqref{e.pde} with initial value $U(0)= U^0$.  Then there is
$m_* \in \N$ such that for every $m \geq m_*$ there exists a solution
$u_m \in \kC([0,T], \kD)$ to the projected system \eqref{e.pde-m} with
initial value $u_m(0) = \P_m U^0$.  Moreover,
\begin{equation}
  \sup_{t \in [0,T]} \norm{U(t)-u_m(t)}{\kY} \to 0
  \label{e.ym-convergence}
\end{equation}
as $m \to \infty$.
\end{theorem}

\begin{proof}
Local existence of a solution $u_m(t)$ of \eqref{e.pde-m} is obvious
since $A_m$ is bounded.  However, we need to show that the interval of
existence is at least $[0,T]$.  We note that the solution can only
cease to exist if $u_m$ leaves the domain $\kD$, so we proceed to
prove \eqref{e.ym-convergence} directly.  Clearly,
\begin{equation}
  \label{e.mild-diff}
  U(t) - u_m(t) = \e^{tA} \Q_m U^0 
  + \int_0^t \e^{(t-s)A} \, (B(U(s)) - B_m(u_m(s))) \, \d s \,.
\end{equation}
Taking the $\kY$-norm and noting that, by \eqref{e.etAEstimate}, there
is $c>0$ such that $\lVert \e^{tA} \rVert_{\kE(\kY)} \leq c$ for $t
\in [0,T]$, we find that
\begin{align}
  & \norm{U(t) - u_m(t)}{\kY}
    \leq c\, \norm{\Q_m U^0}{\kY}
       + c \int_0^t \norm{B(U(s)) - B_m(u_m(s))}{\kY} \, \d s
       \notag \\
  & \leq c \, \norm{\Q_m U^0}{\kY}
       + c \, T \, \sup_{s \in [0,T]} \norm{\Q_m B(U(s))}{\kY} 
       + c \int_0^t \norm{B(U(s)) - B(u_m(s))}{\kY} \, \d s \,.
       \notag \\
  & \leq c \, \norm{\Q_m U^0}{\kY}
       + c \, T \, \sup_{s \in [0,T]} \norm{\Q_m B(U(s))}{\kY} 
       + c \, M_0' \int_0^t \norm{U(s) - u_m(s)}{\kY} \, \d s \,.
  \label{e.u-ym}
\end{align}
Now note that the sequence of functions $f_m(s) = \lVert \Q_m B(U(s))
\rVert_{\kY}$ converges pointwise to zero as $m \to \infty$.
Moreover, since
\[
  \lvert f_m(s_1) - f_m(s_2) \rvert
  \leq \norm{\Q_m (B(U(s_1)) - B(U(s_2)))}{\kY}
  \leq \norm{B(U(s_1)) - B(U(s_2))}{\kY} \,,
\]
the sequence is uniformly equicontinuous.  Hence, by the
Arzel\`a--Ascoli theorem, $f_m$ converges to zero uniformly as $m \to
\infty$.  Thus, applying the Gronwall inequality to \eqref{e.u-ym}, we
see for any $\varepsilon > 0$ there exists a possibly larger $m_*$
such that for $m \geq m_*$, $\lVert U(t) - u_m(t) \rVert_\kY \leq
\varepsilon$ so long as $u_m(t)$ does not leave $\kD$.  Choosing
$\varepsilon < \dist(\{U(s) \colon s \in [0,T]\},\partial \kD)$, we
conclude that $t$ in this estimate may be chosen as large as $T$.
\end{proof}

We now define
\begin{equation}
  \label{e.Rk}
  R_{K+1} = \sup_{U \in \kD_{K+1}} \norm{U}{\kY_{K+1}} \,.
\end{equation}

The following theorem provides higher order bounds for the Galerkin
approximation error of the semiflow.
                                  
\begin{corollary}[Convergence of the projected system -- higher order
error bounds] \label{c.wpp.unif}
Assume \textup{(A)} and \textup{(B1)}.  Let $\delta\in (0,\delta_*]$
be such that $\kD_{K+1}^{-\delta}$ is nonempty.  Then there exists
$m_*$ such that for all $m \geq m_*$ and every mild solution $U \in
\kC([0,T];\kD_{K+1}^{-\delta})$ of the semilinear evolution equation
\eqref{e.pde} there exists a solution $u_m \in \kC([0,T],\kD_K \cap
\kY_{K+1})$ to the projected system \eqref{e.pde-m} with initial value
$u_m(0) = \P_m U^0$ such that
\begin{equation}
  \label{e.wpp.higher}
  \sup_{t \in [0,T]} \norm{U(t)-u_m(t)}{\kY} = O(m^{-K-1}) 
\end{equation}
The order constants in \eqref{e.wpp.higher} depend only on the bounds
afforded by \textup{(B1)}, \eqref{e.etAEstimate}, and \eqref{e.Rk}, on
$\delta$, and on $T$.
 \end{corollary}
\begin{proof}
As in the proof of Theorem~\ref{t.wpp}, we begin with
\eqref{e.mild-diff}.  Here, we apply $\P_m$ and rearrange terms to
obtain the estimate
\begin{equation}
  \label{e.u-ym-unif}
  \norm{U(t) - u_m(t)}{\kY}
  \leq \norm{\Q_m U(t)}{\kY} 
   + c \int_0^t \norm{B(U(s)) - B(u_m(s))}{\kY} \, \d s \,.
\end{equation}
Due to \eqref{e.qm-estimate}, $\tnorm{\Q_m U(\cdot)}{\kY} \leq
R_{K+1}m^{-K-1}$.  The mean value theorem applies so long as $u_m(s)
\in \kD$.  Then, by the Gronwall lemma as before, we find that
\eqref{e.wpp.higher} holds true for all $m \geq m_*$, where we choose
$m_*$ such that $\norm{U(t)-u_m(t)}{\kY}<\delta$ for $t\in [0,T]$ and
$m\geq m_*$ so that indeed $u_m(s) \in \kD$ for $s\in [0,T]$ and
$m\geq m_*$.
\end{proof}
 
Note that Corollary~\ref{c.wpp.unif} with $\kY$ replaced by any of the
$\kY_1, \dots, \kY_K$ readily implies that $\sup_{t \in [0,T]}
\norm{U(t)-u_m(t)}{\kY_j} = O(m^{-K-1-j})$.  However, as $B$ is not
assumed to map from an open subset of $\kY_{K+1}$ to $\kY_{K+1}$,
Theorem \ref{t.wpp} as stated does not apply with $\kY$ replaced by
$\kY_{K+1}$.  However, we can still prove the following.

\begin{corollary}
Under the assumptions of Corollary~\ref{c.wpp.unif}, the following is
true.
\begin{itemize}
\item[(a)] If $N>K+1$, we have
$\sup_{t \in [0,T]} \tnorm{U(t)-u_m(t)}{\kY_{K+1}} \to 0$ as $m \to \infty$. 
\item[(b)] If $N=K+1$, there exists $C>0$ such that $\sup_{t\in [0,T]}
\tnorm{u_m(t)}{\kY_{K+1}} \leq C$.  The bound $C$ depends only on the
bounds afforded by \textup{(B1)}, \eqref{e.etAEstimate}, and
\eqref{e.Rk}, on $\delta$, and on $T$.
\end{itemize}
\end{corollary}

\begin{proof}
We may assume without loss of generality, that $K=0$.  (Otherwise
replace $\kY$ with $\kY_K$.)  Suppose first that $N>K+1$.  Then
Theorem~\ref{t.wpp} applies to the system of evolution equations
\[
  \dot U = A U + B(U) \,, \qquad \dot W = A W + B'(U) W
\]
with initial value $W(0) = A U^0 + B(U^0)$ so that $W(t) = U'(t)$.
Hence,
\[
  \sup_{t \in [0,T]} \norm{U'(t)-u_m'(t)}{\kY} \to 0
\]
as $m \to \infty$.  Since $\sup_{t \in [0,T]} \norm{B(U(t))-\P_m
B(u_m(t))}{\kY} =O(m^{-1})$ by Corollary~\ref{c.wpp.unif} and $A U(t)
= U'(t) - B(U(t))$, we obtain statement (a).

To prove statement (b), integrate $\dot w_m = A w_m + B'(u_m) w_m$
with initial value $w_m(0)=A \P_m U^0 + \P_m B(u_m(0))$ and apply a
standard Gronwall argument as before, noting that the $\kY$-norm of
$u_m(t)$ is bounded uniformly in $m\geq m_*$ by
Corollary~\ref{c.wpp.unif}.  Thus, $\sup_{t\in [0,T]}
\tnorm{u_m'(t)}{\kY} \leq c$ for some $c>0$ depending only on the
bounds afforded by \textup{(B1)}, \eqref{e.etAEstimate}, and
\eqref{e.Rk}, on $\delta$, and on $T$.  This, together with the bound
of $B$ on $\kD$, proves (b).
\end{proof}


\subsection{Regularity of Galerkin truncated semiflow}
\label{ss.notation}

We first introduce some notation.  For Banach spaces $\kX$ and $\kY$,
and $j \in \N_0$, we write $\kE^j(\kY,\kX)$ to denote the vector space
of $j$-multilinear bounded mappings from $\kY$ to $\kX$; we set
$\kE^j(\kX) \equiv \kE^j(\kX,\kX)$.  For Banach spaces $\kX$, $\kY$,
and $\kZ$, and subsets $\kU \subset \kX$, $\kV \subset \kY$, and $\kW
\subset \kZ$, we write
\[
  F \in \kCb^{(\underline{m},n)} (\kU \times \kV; \kW)
\]
to denote a continuous, bounded function $F \colon \kU \times \kV \to
\kW$ whose partial Fr\'echet derivatives $\D_X^i \D_Y^j F(X,Y)$ exist,
are bounded, and are such that the maps
\begin{equation}
  \label{e.derivativemaps}
  (X, Y, X_1,\ldots, X_{i}) 
  \mapsto \D_X^i \D_Y^j  F(X,Y)(X_1, \ldots, X_{i}) 
\end{equation}
are continuous from $\kU \times \kV \times \kX^{i}$ into $\kE^j(\kY,
\kZ)$ for $i = 0, \dots, m$ and $j = 0,\dots, n$ and provided the maps
\eqref{e.derivativemaps} extend continuously to the boundary.  (The
latter is important as we will apply the contraction mapping theorem
to maps in such classes.)  In our setting, $\kV$ will typically be an
interval of time.
 
The following theorem provides regularity of the Galerkin truncated
semiflow with bounds uniform in $m$ under conditions (A) and (B1)
analogous to the regularity result for the semiflow $\Phi$ in
\cite[Theorem~2.4]{OW2010}.

\begin{theorem}[Regularity of the Galerkin truncated semiflow, local
version] \label{t.local-diff}
Assume \textup{(A)} and \textup{(B1)}.  Choose $R\in (0,\delta_*]$
small enough such that $\kD_K^{-R} \neq \emptyset$ and pick $U^0 \in
\kD_K^{-R}$.  Then there is $T_*=T_*(R,U^0)>0$ and $m_*(R,U^0)\in\N$
such that for $m\geq m_*$ there exists a semiflow $\Phi^t_m$ of
\eqref{e.pde-m} of class
\begin{subequations}
\begin{equation}
  \label{e.T-diff-gen}
  \Phi_m \in \bigcap_{\substack{j+k \leq N \\ \ell \leq k \leq K}}
  \kCb^{(\underline j, \ell)} (B_{R/2}^{\kY_K}(U^0) \times [0,T_*]; 
  \kB_R^{\kY_{k-\ell}}(U^0)) \,.
\end{equation}
The bounds on $\Phi_m$ and $T_*$ depend only on the bounds afforded by
\textup{(B1)} and \eqref{e.etAEstimate}, on $R$, and on $U^0$.  In
particular,
\begin{equation}
  \label{e.T-diff-gen.b}
  \Phi_m \in 
    \kCb^K (B_{R/2}^{\kY_K}(U^0) \times [0,T_*]; \kB_R^\kY(U^0)) \,.
\end{equation}
\end{subequations}
\end{theorem}
  
\begin{proof}
The proof is an application of Theorem~\ref{t.cm-stab} (a) on
contraction mappings on a scale of Banach spaces.  We consider $\Pi$
from \eqref{e.PiFlow} and write the corresponding contraction map for
the semiflow $\Phi_m$ of the projected system as
\begin{equation}
  \label{e.Pim}
  \Pi_m(W; U, h) = \P_m \Pi (W; \P_m  U, h) \,.
\end{equation}
We replace $N$ from Theorem~\ref{t.cm-stab} by $N-1$, set $\mu=T$,
$\kI = (0, T_*)$, $\kX=\kY_K$ with $\kU \equiv \kU_K =
\kB_{R/2}^{\kX}(U^0)$, $w = W$, and $\kZ_j = \kC([0,1];\kY_j)$ with
$\kW_j =\kC([0,1]; \kB^{\kZ_j}_{R}(U^0))$ for $j=0, \dots, K$.  We now
show that the contraction maps $\Pi_m$ satisfy conditions (i) and (ii)
of Theorem~\ref{t.cm-stab} for some $T_*(R,U^0)>0$ and $m\geq
m_*(R,U^0)$.  We first show that $\Pi_m$ maps each $\kW_0, \dots,
\kW_K$ into itself.  We estimate, using (B1) and
\eqref{e.etAEstimate}, that
\begin{align}
  \norm{\Pi_m(W;U,T) - U^0}{\kY_j}
  & \leq \norm{\e^{\tau T A}  U^0 - U^0}{\kY_j} + 
         \norm{\e^{\tau T A}(\P_m U^0 - U^0)}{\kY_j} \notag \\
  & \quad 
    + \e^{\omega T} \, R/2 + T \,  \, \e^{\omega T} \, M_j \,.  
  \label{e.localEx}
\end{align}
Choosing $T_* = T_*(R,U^0)>0$ sufficiently small, the second line of
\eqref{e.localEx} can be made less than $3R/4$.  Moreover, for a
possibly smaller value of $T_*$, there exists $m_* = m_* (R, U^0)$
such that for all $m \geq m_*$, $T\in [0,T_*]$, and $\tau\in[0,1]$ the
first line of \eqref{e.localEx} is less than $R/4$.  Then, the right
hand side of \eqref{e.localEx} is less than $R$ which proves that
$\Pi_m$ maps back into $\kW_j$.  Assumption (B1) and (A) then imply
condition (i) of Theorem~\ref{t.cm-stab}.  To show condition (ii) we
estimate, noting that $N>K$ by (B1), that
\begin{equation}
  \label{e.flow_contraction}
  \norm{\D_W \Pi_m(W;U,T)}{\kE(\kCb([0,1];\kY_j))}
  \leq T \,  \e^{\omega T} \, M'_j \,,
\end{equation}
so that $\Pi_m$ is a uniform contraction for all $m\geq m_*$, $U \in
\kU$, $W \in \kW_j$, and $T \in \kI = (0,T_*)$ for every $j=0,\ldots,
K$ with a possibly smaller value of $T_*$.
 
Hence, $\Pi_m$ satisfies conditions (i) and (ii) of
Theorem~\ref{t.cm-stab} with bounds and contraction constants which
are uniform in $m\geq m_*$ so that Theorem~\ref{t.cm-stab} (a) implies
that $\Phi_m$ is of class \eqref{e.T-diff-gen}.  The simplified
special case \eqref{e.T-diff-gen.b} is a direct consequence of
Lemma~\ref{l.cm-scale}.
\end{proof}
 
Theorem~\ref{t.local-diff} does not guarantee that $m_*$ and $T_*$ can
be chosen uniformly over $\kD$.  The following theorem states that
such uniformity can be obtained, however, over domains other than
balls at the expense of stepping up on the scale of Hilbert spaces.
The situation is analogous to that for the semiflow $\Phi$; see
\cite[Theorem~2.6 and Remark~2.8]{OW2010}.
 
\begin{theorem}[Regularity of Galerkin truncated semiflow, uniform
version] \label{t.local-diff-unif}  
Assume \textup{(A)} and \textup{(B1)}.  Choose $\delta \in
(0,\delta_*]$ small enough such that $\kD_{K+1}^{-\delta} \neq
\emptyset$.  Then there exists $T_*=T_*(\delta)>0$ and
$m_*(\delta)\in\N$ such that for $m\geq m_*$ the semiflow $(U,t)
\mapsto \Phi^t_m(U)$ of \eqref{e.pde-m} satisfies \eqref{e.T-diff-gen}
with bounds which are uniform for all $U^0 \in \kD_{K+1}^{-\delta}$
with $R=\delta$, and such that
\begin{subequations}
\begin{equation}
  \label{e.T-diff-gen-unif}
    \Phi_m \in \bigcap_{\substack{j+k \leq N \\ \ell \leq k \leq K+1}}
  \kCb^{(\underline j, \ell)} (\kD_{K+1}^{-\delta} \times [0,T_*]; 
  \kY_{k-\ell}) \,
\end{equation}
with bounds which are uniform in $m\geq m_*$.  The bounds on $\Phi_m$,
$m_*$ and $T_*$ depend only on the bounds afforded by \textup{(B0)}
rsp.\ \textup{(B1)}, \eqref{e.etAEstimate}, and \eqref{e.Rk}, and on
$\delta$.  Moreover, $\Phi_m$ maps into $\kD_K$ and, when $N>K+1$,
\begin{equation}
  \label{e.T-diff-gen-unif.b}
  \Phi_m \in \kCb^{K+1}(\kD_{K+1}^{-\delta}\times [0,T_*]; \kY)
\end{equation}
\end{subequations}
with corresponding uniform bounds.
\end{theorem}
 
\begin{proof}
We continue to work in the setting introduced in the proof of
Theorem~\ref{t.local-diff}.  Here, we need to verify that the
conditions of Theorem~\ref{t.cm-stab} are satisfied uniformly in $U^0
\in \kD_{K+1}^{-\delta}$ for both $\Pi_m$ and $\tilde \Pi_m$.  First,
due to \textup{(B1)}, each of the $\Pi_m$ is well-defined as a map
from $\kW_j\times\kU\times\kI$ into $\kZ_j$ for $U^0 \in
\kD_{K+1}^{-\delta}$ and has the required regularity.  To show that
there is $m_*(\delta)$ such that $\Pi_m$ maps $\kW_0, \dots, \kW_K$
back into itself, we apply \eqref{e.localEx} for every $U^0 \in
\kD_{K+1}^{-\delta}$.  We bound the first line on the right-hand side of
\eqref{e.localEx} by
\begin{multline}
  \norm{\e^{\tau T A} U^0 - U^0}{\kY_j} 
  + \norm{\e^{\tau T A}(\P_m U^0 - U^0)}{\kY_j} \\
  \leq \max_{t \in [0,T]}  
       \bigl(
         T \, \norm{A\e^{t A} U^0 }{\kY_j}
         + \norm{\e^{t A}\Q_m U^0 }{\kY_j}
       \bigr)
  \leq \e^{\omega T} \, R_{K+1} \, (T + 1/m) \,,
  \label{e.exptA-unif}
\end{multline}
where $R_{K+1}$ is defined in \eqref{e.Rk} and $j=0,\ldots, K$.
Inserting this estimate into \eqref{e.localEx}, we see that we can
choose $T_*>0$ small enough such that $\Pi_m(\,\cdot\,; U, T)$ maps
$\kB_{{R}}^{\kY_j}(U^0)$ with $R=\delta$ into itself for all $U^0 \in
\kD_{K+1}^{-\delta}$, $U\in\kU$, $m\geq m_*$, $T \in [0,T_*]$, and
$j=0,\ldots, K$.  Hence, $\Pi_m$ satisfies the conditions of
Theorem~\ref{t.cm-stab} with bounds which are uniform in $m\geq m_*$,
$T \in (0,T_*)$, and $U^0 \in \kD_{K+1}^{-\delta}$. This shows that
\eqref{e.T-diff-gen} holds uniformly for $U^0 \in \kD_{K+1}^{-\delta}$
and $m\geq m_*(\delta)$.

Next, we show that 
\begin{equation}
  \label{e.Aphim-unif}
  A \Phi_m \in \bigcap_{\substack{j+k \leq N \\ \ell \leq k \leq K}}
    \kCb^{(\underline j, \ell)} (\kD_{K+1}^{-\delta} \times [0,T_*]; 
    \kY_{k-\ell})
\end{equation}
with uniform bounds in $m\geq m_*$.  Consider the linear fixed point
equation $\tilde{W}_m = \tilde \Pi_m(\tilde{W}_m; U, T)$ with
\begin{align}
 \tilde \Pi_m(\tilde{W}_m; U, T)(\tau) 
  & = \e^{\tau T A} \, (A \P_m U + B(W_m(0))) - B(W_m(\tau)) 
      \notag \\
  & \quad 
      + T \int_0^{\tau} \e^{(\tau-\sigma)TA} \, 
        \D B_m(W_m(\sigma)) (\tilde{W}_m(\sigma) 
        + B(W_m(\sigma))) \, \d\sigma 
  \label{e.tildePim}
\end{align}
where $W_m(U,T)(\tau) = \Phi_m^{\tau T}(U)$.  Integrating the right
hand side of \eqref{e.PiFlow} by parts, replacing $B$ with $B_m$ and
$U^0$ by $\P_m U^0$ we see that the fixed point $\tilde{W}_m$ of
$\tilde\Pi_m$ satisfies $\tilde{W}_m=AW_m$ in $\kZ_j$ for $j=0,
\ldots, K$.  We consider $\tilde\Pi_m$ with $\kU$, $\kZ_j$ and $\kI$
as before, and set $\kW_j = \kC([0,1];\kB_r^{\kZ_j}(0))$ with $r>0$
large enough that $\tilde\Pi_m(\cdot, U,T)$ maps $\kW_j$ into itself
for $m\geq m_*$, $U\in \kU$, $T \in \kI$.  Since $\Phi_m$ is of class
\eqref{e.T-diff-gen}, Lemma~\ref{l.diff} (a) and
Lemma~\ref{l.DiffhatPiSigma} (a) imply that $\tilde\Pi_m$ satisfies
the conditions of Theorem~\ref{t.cm-stab} with $N$ replaced by
$N-2$. Therefore Theorem~\ref{t.cm-stab} (a) applies and proves
\eqref{e.Aphim-unif}.
 
Moreover, $B_m \circ \Phi_m$ is of class \eqref{e.Aphim-unif} with
uniform bounds for $m\geq m_*$ due to the chain rule,
Lemma~\ref{l.diff} (a) and the fact that $\Phi_m$ is of class
\eqref{e.T-diff-gen} with uniform bounds in $U^0$ and $m\geq m_*$.  We
conclude that $\partial_t \Phi_m = A \Phi_m +B_m\circ \Phi_m $ is also
of class \eqref{e.Aphim-unif} with uniform bounds for $m\geq m_*$.

Finally, as both $A \Phi_m $ and $\partial_t \Phi_m $ are of class
\eqref{e.Aphim-unif}, Lemma~\ref{l.technical1} implies
\eqref{e.T-diff-gen-unif}.  The simplified special case
\eqref{e.T-diff-gen-unif.b} is a direct consequence of
Lemma~\ref{l.cm-scale}.
\end{proof}


\subsection{Accuracy of derivatives of Galerkin truncated semiflow}

In this section, we consider how the perturbation of the contraction
map $\Pi$ introduced by the projection of the evolution equation
\eqref{e.pde} onto the subspace $\P_m \kY$ propagates into derivatives
of the resulting semiflow.
 
As before, we consider a local and a uniform version of each result;
the scales we use are defined, separately for the two cases, as
follows.  In the local version, we follow the setting of
Theorem~\ref{t.local-diff}, where we consider initial data
\begin{subequations}
  \label{e.localWSpaces}
\begin{equation}
  \label{e.localUSpace}
  U \in \kU \equiv \kB^{\kX}_{R_*}(U^0)
  \qquad \text{where} \qquad
  \kX =\kY_K \,.
\end{equation}
The semiflows are considered as maps
\begin{equation}
  \label{e.localPsiPhiSpace}
  \Phi^t, \Phi^t_m  \colon \kB^{\kY_K}_{R_*}(U^0) \to \kY_j
  \quad\mbox{for} \quad j = 0, \dots, K,
\end{equation}
\end{subequations}  
where $m \geq m_*(\delta,U^0)$ and $R_*=R/2$.
 
In the uniform version, we follow the setting of
Theorem~\ref{t.local-diff-unif} where we consider initial data
\begin{subequations}
  \label{e.uniformWSpaces}
\begin{equation}
  U \in \kU = \kD^{-\delta}_{K+1} \subset \kX \equiv \kY_{K+1} \,.
\end{equation}
The semiflows are considered as maps
\begin{equation}
  \label{e.uniformPsiSpaces}
  \Phi^t, \Phi^t_m  \colon \kD^{-\delta}_{K+1} \to \kY_{j}
 \quad\mbox{ for}\quad j = 0, \dots, K+1,
\end{equation}
\end{subequations}
for some fixed $\delta>0$ where $m \geq m_*(\delta)$.
   
To handle the complexity of these estimates it is useful to define
norms on the various objects that contain all combinatorially possible
orders of differentiation and scale rungs subject to certain relevant
side constraints.  The need to consider such norms arises through the
implicit nature of the definition of the semiflow and the use of the
chain rule.  Here, any attempt to estimate a particular derivative on
a particular rung of the scale will produce terms of all intermediate
orders of differentiation and scale rungs.  We therefore estimate all
derivatives at once.

We have to deal with two different types of objects: contraction maps
which are functions of three arguments whose corresponding norms are
denoted $\Norm{\, \cdot \,}{}$ and semiflows which are functions of
two arguments whose corresponding norms are denoted $\tnorm{\, \cdot
\,}{}$.

In our setting, there are two natural global parameters, namely $N$,
the degree of differentiability of the nonlinearity, and $K$, the
number of rungs on our scale as defined in condition (B1).  Two more
characteristic parameters are needed.  First, the \emph{loss index}
$S$ which forces the image of the map be estimated at least $S$ rungs
down the scale.  We will see that a loss of $S$ scale rungs translates
into $O(m^{-S})$-smallness of the perturbation caused by the projector
$\P_m$.  Second, a \emph{lowest rung index} $L$ which forces the
estimation of the image of the map to occur at least $L$ rungs up from
the bottom of the scale.  This leads us to define a four parameter
family of norms for functions $\Pi = \Pi(w;u,\mu)$ mapping $\kW_{k+S}
\times \kU \times \kI$ to $\kZ_{k-\ell}$,
\begin{subequations}
\begin{gather}
  \label{e.norm4Pi}
  \Norm{\Pi}{N,K,L,S} 
  = \max_{\substack{i+j+k \leq N-S \\ 
          L+\ell \leq k \leq K-S}}
    \norm{\D_w^i \D_u^j \partial_\mu^\ell \Pi}{\kL_\infty
          (\kW_{k+S} \times \kU \times \kI;
           \kE^i(\kZ_{k+S},\kE^j(\kX; \kZ_{k-\ell})))}
\end{gather}
for $0 \leq L \leq K-S \leq N-S$.  When studying semiflows, we
identify $w=W$, $u=U$, $\mu=T$, $\kI = (0,T_*)$, $\kZ_j =
\kC([0,1];\kY_j)$, $\kU=\kB_{R/2}^{\kY_j}(U^0)$, $\kX=\kY_K$, and
$\kW_j = \kC([0,1];\kB_{R}^{\kY_j}(U^0))$, and use $\Pi$ defined by
\eqref{e.PiFlow}.  We abbreviate
\begin{gather}
  \label{e.norm3Pi} 
  \Norm{\Pi}{N,K,L} = \Norm{\Pi}{N,K,L,0} \,, \\
  \label{e.norm2Pi} 
  \Norm{\Pi}{N,K} = \Norm{\Pi}{N,K,0,0} \,.
\end{gather}
Functions $w=w(u,\mu)$ mapping $\kU \times \kI$ to $\kZ_j$ are
equipped with the three parameter family of norms
\begin{gather}
  \label{e.norm3w}
  \norm{w}{N,K,L} 
  = \max_{\substack{j+k \leq N \\ 
          L+\ell \leq k \leq K}}
    \norm{\D_u^j \partial_\mu^\ell w}{\kL_\infty
          (\kU \times \kI; \kE^j(\kX; \kZ_{k-\ell}))} 
\intertext{for $0 \leq L \leq K \leq N$, where we abbreviate}
  \label{e.norm2w} 
  \norm{w}{N,K} = \norm{w}{N,K,0} \,.
\end{gather}
\end{subequations}
With this notation, a function $(u,\mu) \mapsto \Pi(u,\mu)$ that does
not depend on $w$ satisfies
\begin{equation}
  \label{e.3Norm2Norm}
  \Norm{\Pi}{N,K,L,S} = \norm{\Pi}{N-S,K-S,L} \,.
\end{equation}

The next pair of results concerns the stability of the semiflow and
its derivatives under spectral truncation.

\begin{theorem}[Projection error for the semiflow, local version]
\label{t.flow-projection-error-scale} 
Assume \textup{(A)} and \textup{(B1)} and $R \in (0,\delta_*]$ small
enough such that $ \kD_K^{-R}\neq \emptyset$ and pick $U^0
\in\kD_K^{-R}$.  Let $T_*=T_*(R,U^0)$ and $m_*=m_*(R,U^0)$ be as in
Theorem~\ref{t.local-diff}.  Then, for every $0 \leq P \leq K$,
\begin{equation}
  \norm{\Phi - \Phi_m}{N-P-1,K-P} = O(m^{-P})
  \label{e.phi-m-s}
\end{equation}
where the norm in \eqref{e.phi-m-s} is defined with respect to the
spaces \eqref{e.localWSpaces}.  The order constants depend only on the
bounds afforded by \textup{(B1)} and \eqref{e.etAEstimate}, on $U^0$,
and on $R$.
\end{theorem}

\begin{proof}
We apply Theorem~\ref{t.cm-stab} to obtain a bound on $\tnorm{\Phi -
\Phi_m}{N-P-1,K-P}$ in terms of $\Norm{\Pi - \Pi_m}{N-1,K,0,P}$, with
$\kZ_k$, $\kX$ etc.\ specified above.  We already verified conditions
(i) and (ii) of Theorem~\ref{t.cm-stab} in the proof of
Theorem~\ref{t.local-diff}.  Thus, in order to prove
\eqref{e.phi-m-s}, it suffices to show that
\begin{equation}
  \label{e.pi-difference}
  \Norm{\Pi-\Pi_m}{N-1,K,0,P} = O(m^{-P}) \,.
\end{equation} 
Writing
\begin{equation}
  (\Pi - \Pi_m)(W; U, T)(\tau)
  = G_m(U,T)(\tau) + I_m(W,T)(\tau) 
  \label{e.flow-splitting}
\end{equation}
with
\begin{gather*}
  G_m(U,T)(\tau) 
  = \Q_m \e^{T \tau A} U \,, \\
  I_m(W,T)(\tau) 
  = \Q_m \int_{0}^\tau T \e^{(\tau-\sigma)TA} \,
    B(W(\sigma)) \, \d\sigma \,,
\end{gather*}
we apply Lemma~\ref{l.diff} to both terms on the right-hand side of
\eqref{e.flow-splitting} in different ways.  For the first term, since
$G_m$ does not depend on $W$, we can apply Lemma~\ref{l.diff} with
$\Pi=\id$, $\Sigma=\P_m$, and $v(U,T)(\tau) = w(U,T)(\tau) = \e^{\tau
TA} U$, so that there exists $c_1>0$ such that
\[
  \Norm{G_m}{N-1,K,0,P}
  = \norm{G_m}{N-1-P,K-P}
  \leq c_1 \, \Norm{\Q_m}{N-1,K,0,P}
  = O(m^{-P}) \,,
\]
where the first equality is due to \eqref{e.3Norm2Norm}.  Here, and
further below, we implicitly make use of estimate
\eqref{e.etAEstimate} on the bound of the linear semigroup and
estimate \eqref{e.qm-estimate} on the Galerkin remainder.

To apply Lemma~\ref{l.diff} to the second term on the right-hand side of
\eqref{e.flow-splitting}, we identify $N$ and $K$ there with $N-1$ and
an arbitrary $\kappa \in P, \dots, K$ here.  Then, setting $\Pi =
\id$, $\Sigma=\P_m$, $u=W \in \kU_\kappa \equiv \kC([0,1];
\kB_{R}^{\kY_\kappa}(U^0))$, and
\[
  v(u,T)(\tau) = w(u,T)(\tau) = 
  \int_0^\tau T \e^{(\tau-\sigma)TA} B(W(\sigma)) \, \d\sigma \,,
\]
Lemma~\ref{l.diff} asserts that there exists $c_2 > 0$ such that
\[
  \norm{I_m}{N-1-P,\kappa-P}
  \leq c_2 \, \Norm{\Q_m}{N-1,\kappa,0, P} 
  = O(m^{-P}) \,.
\]
Then, by Lemma~\ref{l.technical3}, with $N$ replaced by $N-1$ and
$S=P$,
\[
  \Norm{I_m}{N-1,K,0,P} = O(m^{-P}) \,.
\]
The constants $c_1$ and $c_2$ depend only on the bounds on $B$ from
(B1) and the bounds from \eqref{e.etAEstimate}.  Altogether, this
verifies \eqref{e.pi-difference}, thus concludes the proof.
\end{proof}

\begin{theorem}[Projection error for the semiflow, uniform version]
\label{t.flow-projection-error-scale-unif} 
Assume \textup{(A)} and \textup{(B1)} with $N>K+1$, and let $\delta
\in (0,\delta_*]$ small enough such that $\kD_{K+1}^{-\delta}$ is
nonempty.  Let $T_*=T_*(\delta)>0$ and $m_*=m_*(\delta)$ be as in
Theorem~\ref{t.local-diff-unif}.  Then, for every $0 \leq P \leq K+1$,
\begin{equation}
  \norm{\Phi - \Phi_m }{N-P-1,K+1-P} = O(m^{-P}) \,,
  \label{e.phi-m-s-unif}
\end{equation} 
where the norm in \eqref{e.phi-m-s-unif} is defined with respect to
the spaces \eqref{e.uniformWSpaces}.  The order constants depend only
on $\delta$, and on the bounds afforded by \textup{(B1)},
\eqref{e.etAEstimate}, and \eqref{e.Rk}.
\end{theorem}

\begin{proof}
First, we show that
\begin{equation}
  \label{e.phi-m-s-unif-YK}
  \norm{\Phi - \Phi_m}{N-P-1,K-P} = O(m^{-P}) \,,
\end{equation}
for $0\leq P\leq K$ on the scale $\{\kZ_j\}_{j=0,\ldots, K}$.  To this
end, define $\Pi$ and $\Pi_m$, $\kU$, $\kW_j$ etc., with $R=\delta$ as
in the proof of Theorem~\ref{t.flow-projection-error-scale}.
 
We have already shown in the proof of Theorem~\ref{t.local-diff-unif}
that conditions (i) and (ii) of Theorem~\ref{t.cm-stab} hold uniformly
in $m\geq m_*$ and $U^0 \in \kD_{K+1}^{-\delta}$.  Moreover,
\eqref{e.pi-difference} holds true uniformly in $U^0 \in
\kD_{K+1}^{-\delta}$.  This is easily verified by checking that each
of the estimates in the proof of
Theorem~\ref{t.flow-projection-error-scale} holds uniformly under the
conditions of Theorem~\ref{t.flow-projection-error-scale-unif}.
Hence, Theorem~\ref{t.cm-stab} (b) implies \eqref{e.phi-m-s-unif-YK}.
 
Next, we apply Theorem~\ref{t.cm-stab} to obtain a bound on
$\Norm{\tilde\Pi - \tilde\Pi_m}{N-2,K,0,P}$, where $\tilde\Pi_m$ is
from \eqref{e.tildePim} and $\tilde\Pi$ is defined correspondingly.
We have shown in the proof of Theorem~\ref{t.local-diff-unif} that
$\tilde\Pi_m$ (and hence also $\tilde\Pi$) satisfy the conditions of
Theorem~\ref{t.cm-stab} uniformly for $m\geq m_*$.  Estimating each
term of the corresponding analogue to \eqref{e.flow-splitting} via
Lemma~\ref{l.diff} and Lemma~\ref{l.DiffhatPiSigma}, we find that
$\Norm{\tilde\Pi - \tilde\Pi_m}{N-2,K,0,P} = O(m^{-P})$.  Then,
Theorem~\ref{t.cm-stab} (b) implies that $\tnorm{A\Phi -
A\Phi_m}{N-P-2,K-P} = O(m^{-P})$ so that, for $0\leq P\leq K$,
\begin{equation}
  \label{e.APhiphim}
  \norm{\Phi - \Phi_m}{N-P-1,K-P+1,1} = O(m^{-P}) \,.
\end{equation}

Finally, we prove that
\begin{equation}
  \label{e.dtPhi-phim}
  \norm{\partial_t \Phi -\partial_t \Phi_m}{N-P-2,K-P} = O(m^{-P}) \,.
\end{equation}
We note that $\partial_t (\Phi - \Phi_m) = A (\Phi-\Phi_m) + B \circ
\Phi- B_m \circ \Phi_m$, where the required bound on the first term on
the right-hand side has already been established.  For the second term, we use
Lemma~\ref{l.diff} with $\Pi = B$, $\Sigma=B_m$, $\mu=t \in
\kI=(0,T_*)$, $u=U^0 \in \kU = \kD_{K+1}^{-\delta} $,
$w(U^0,t)=\Phi^t(U^0)$, $v(U^0,t)=\Phi_m^t(U^0)$, $\kX=\kY_{K+1}$,
$\kZ_j = \kY_j$, and $\kW_j = \kD_j$ for $j=0,\ldots, K$.  Hence,
there exists a constant $c_1$ such that
\[
  \norm{B\circ \Phi - B_m \circ \Phi_m}{N-2-P,K-P} 
  \leq c_1 \, \Norm{\Q_m B}{N-2,K,0,P} 
     + c_1 \, \norm{\Phi - \Phi_m}{N-2-P,K-P} \,.
\]
To estimate $\Norm{\Q_m B}{N-2,K,0,P}$, we apply Lemma~\ref{l.diff}
for each $\kappa \in P, \dots, K$ with $\Pi = \id$, $\Sigma=\P_m$,
$u=W$, $v(u,h) = w(u,h) = B(W)$, $\kU$ replaced by $\kU_\kappa \equiv
\kW_\kappa$, $N$ replaced by $N-1$, and $K$ replaced by $\kappa$.
Hence, there is some constant $c_2$ such that
\[
  \norm{\Q_m  B(W)}{N-2-P,\kappa-P}
  \leq c_2 \, \Norm{\Q_m}{N-2,\kappa,0, P} = O(m^{-P}) \,.
\]
Then, Lemma~\ref{l.technical3} with $N$ replaced by $N-1$ and $S=P$
implies
\[
  \Norm{\Q_m  B}{N-2,K,0,P} = O(m^{-P}) \,.
\]
Altogether, this proves \eqref{e.dtPhi-phim}.  Finally,
\eqref{e.phi-m-s-unif} follows from Lemma~\ref{l.technical1} with
$L=0$ due to \eqref{e.phi-m-s-unif-YK}, \eqref{e.APhiphim}, and
\eqref{e.dtPhi-phim}.
\end{proof}


\section{A-stable Runge--Kutta methods under Galerkin truncation}
\label{s.rk}

In this section, we study a class of A-stable Runge--Kutta methods
that are well-defined when applied to the semilinear PDE
\eqref{e.pde} under assumptions (A) and (B1).  We prove regularity of
spectral Galerkin approximations of such methods uniformly in the
spatial resolution and derive estimates for the approximation error.
The class of methods we consider is the same as in \cite{OW2010}.

Applying an $s$-stage Runge--Kutta method to the semilinear evolution
equation \eqref{e.pde}, we obtain
\begin{subequations} \label{e.as-rk}
\begin{align}
  W   & = {U}^0 \, \1 + h \, {\a} \, 
          \bigl( {A} {W} + {B}(W) \bigr) \,,
          \label{eq:RKStagesIter} \\
  U^1 & = U^0 + h \, {\b}^T \, 
          \bigl( {A} W + {B}(W) \bigr) \,. 
          \label{eq:RKUpdate}
\end{align}
\end{subequations}
For $U \in \kY$ we write
\[
  \1 \,U= 
  \begin{pmatrix}
   U \\ \vdots \\ U
  \end{pmatrix} \in \kY^s \,,
  \quad
  W = 
  \begin{pmatrix}
    W^1 \\ \vdots \\ W^s
  \end{pmatrix} \,,
  \quad
  {B}(W) = 
  \begin{pmatrix}
    B(W^1) \\ \vdots \\ B(W^s)
  \end{pmatrix} \,,
\]
where $W^1, \dots, W^s$ are the stages of the Runge--Kutta method,
\[
  (\a W)^i = \sum_{j=1}^s \a_{ij} \, W^j \,,
  \qquad
  \b^T W  = \sum_{j=1}^s \b_{j} \, W^j \,,
\]
and $A$ acts diagonally on the stages, i.e., $({A}W)^i = A W^i$ for
$i=1,\dots,s$. 
 
A more suitable form, required later, is achieved by rewriting
\eqref{eq:RKStagesIter} as
\begin{equation}
  \label{eq:newRKStagesIter}
  W = \Pi(W; U, h) 
    \equiv (\id - h  {\a} {A} )^{-1} \, 
    (\1 U + h  {\a}  {B}(W)) 
\end{equation}
and  
\begin{align}
  \Psi^h(U)  = \mS(hA) U + h {\b}^T \, (\id - h \a A)^{-1} \, B(W(U,h)) \,,
  \label{eq:NewRKUpdate}
\end{align}
where $S$ is the so-called \emph{stability function}
\begin{equation}
  \label{e.SzRK}
  \mS(z) = 1 + z \b^T \, (\id - z \a)^{-1} \, \1 \,.
\end{equation}

We now make a number of assumptions on the method and its interaction
with the linear operator $A$.  First, we assume that the method is
A-stable in the sense of \cite{LubOst93}.  Setting $\C^- = \{ z \in \C
\colon \Re z \leq 0\}$, the conditions are as follows.
\begin{itemize}
\item[(RK1)] The stability function \eqref{e.SzRK} is bounded with
$\vert \mS(z) \rvert \leq 1$ for all $z \in \C^-$.
\item[(RK2)] The $(s, s)$-matrices $\id - z \a$ are invertible for
all $z \in \C^-$.
\end{itemize}
Sometimes, we will also assume that $\a$ is invertible.
Gauss--Legendre Runge--Kutta methods satisfy conditions \textup{(RK1)}
and \textup{(RK2)} with $\a$ invertible \cite[Lemma 3.6]{OW2010}.

We now summarize the analytic properties of the operators appearing in
\eqref{eq:newRKStagesIter} and \eqref{eq:NewRKUpdate}, where we use
the convention $\tnorm{W}{\kY_\ell^s} = \max_{j=1}^s
\tnorm{W^j}{\kY_\ell}$.  Proofs can be found in
\cite[Section~3.2]{OW2010}.

\begin{lemma} \label{l.RKShA}
Assume \textup{(RK1)}, \textup{(RK2)}, and \textup{(A)}.  Then there
exist $h_*>0$, $\Lambda \geq 1$, $\sigma\geq 0$, and $c_\mS \geq 1$
such that 
\begin{subequations}
  \label{e.RKShA.1}
\begin{gather} 
  \norm{(\id - h\a A)^{-1}}{\kY^s \to \kY^s} \leq \Lambda \,,  
  \label{e.RKShA.a} \\
  \norm{\mS(hA)}{\kY^s \to \kY^s} \leq 1 + \sigma \, h \leq c_\mS \,,
  \label{e.RKShA.c} \\
  \norm{h \a A (\id - h \a A)^{-1}}{\kY^s \to \kY^s} \leq 1 + \Lambda \,,
  \label{e.RKShA.b} 
  \end{gather}
\end{subequations}
for all $h \in [0,h_*]$.  Moreover, for any $\ell, n, \in \N_0$,
\begin{subequations}
\begin{gather}
  (W, h) \mapsto (\id - h \a A)^{-1} W
  \text{ is a map of class }
  \kCb^{(\underline{n},\ell)}(\kB_1^{ \kY^s_\ell }(0)\times[0,h_*]; \kY^s)
  \,, \label{e.RKShA.A} \\
  (W, h) \mapsto h \a A (\id - h \a A)^{-1} W
  \text{ is a map of class }
  \kCb^{(\underline{n},\ell)}(\kB_1^{ \kY^s_\ell }(0)\times[0,h_*]; \kY^s)
  \,, \label{e.RKShA.B} \\
  (W,h) \to h (\id-h\a A)^{-1}W
  \text{ is a map of class }
  \kCb^{(\underline{n},\ell+1)}(\kB_1^{ \kY^s_\ell }(0)\times [0,h_*];\kY^s)
  \,, \label{e.RKShA.B2} \\
\intertext{and}
  (U, h) \mapsto \mS(hA) U
  \text{ is a map of class }
  \kCb^{(\underline{n},\ell)}(\kB_1^{ \kY_\ell }(0)\times[0,h_*]; \kY) \,.
  \label{e.RKShA.C}
\end{gather}
  \label{e.RKShA.2}
\end{subequations}
\end{lemma}


\subsection{Regularity of Galerkin truncated time-discretization}

Let $W_m(U^0,h)$ denote the stage vector, with $W_m^j(U^0,h)$ for $j=1,
\dots, s$ its components, and $\Psi_m^h(U^0,h)$ denote the numerical
time-$h$ map obtained by applying an $s$-stage Runge-Kutta method to
the projected semilinear evolution equation \eqref{e.pde-m} with
initial value $u_m(0) = \P_m U^0$.  Their regularity, with uniform
bounds in the spatial resolution $m$, is stated in the following
theorems which, again in local and uniform version, provide the
analogue to what is known for the time-$h$ map $\Psi$ in the
time semi-discrete case \cite[Theorems~3.14, 3.15, and
Remark~3.17]{OW2010}.
 
\begin{theorem}[Regularity of Galerkin truncated numerical method,
local version]
\label{t.h-diff} 
Assume that the semilinear evolution equation \eqref{e.pde} satisfies
conditions \textup{(A)} and \textup{(B1)}.  Apply a Runge--Kutta
method $\Psi$ subject to conditions \textup{(RK1)} and \textup{(RK2)}
to it.  Choose $R\in (0,\delta_*]$ such that
 $\kD^{-R}_K \neq \emptyset$ and pick $U^0 \in \kD^{-R}_K$.  Let $R_*=R/(2\max\{c_\mS,\Lambda\})$ with
$\Lambda$ and $c_\mS$ from \eqref{e.RKShA.1}.  Then there exists $m_*
= m_*(R, U^0)$ and $h_*=h_*(R,U^0)>0$, such that for $m\geq m_*$ there
exists a stage vector $W_m$ whose components $W_m^i$ as well as the
numerical time-$h$ map $\Psi^h_m(U,h)=\Psi^h_m(U)$ are of class
\eqref{e.T-diff-gen} with $T_*$ replaced by $h_*$.  The bounds on
$W_m$, $\Psi_m$ and $h_*$ are independent of $m$ and depend only on
the bounds afforded by \textup{(B1)} and \eqref{e.RKShA.1}, on the
coefficients of the method, $R$, and $U^0$.
\end{theorem}
 
\begin{proof}
As in the proof of Theorem~\ref{t.local-diff} on the regularity of the
semiflow $\Phi_m$, we apply Theorem~\ref{t.cm-stab} (a) on contraction
mappings on a scale of Banach spaces.  Here we set 
\[
 w = W, \kZ_j =\kY_j^s, ~~\mbox{and}~~ \kW_j = \kB^{\kY_j^s}_{R }(\1 U^0) ~~\mbox{for}~~
 j=0, \dots,K. 
\]
 We further identify $\mu=h$, $\kI = (0, h_*)$, and $\kX=\kY_K$
with $\kU \equiv \kU_K = \kB^{\kY_K}_{R_*}(U^0)$.

The map $\Pi$ for the stage vector $W$ is defined by
\eqref{eq:newRKStagesIter}; we write the corresponding map
$\Pi_m(W,U,h)= \P_m \Pi(W,\P_m U, h)$, analogous to \eqref{e.Pim} for
the semiflow.  We now show that the differentiability assumption on
$B$ is such that $\Pi_m$ satisfies the conditions of Theorem
\ref{t.cm-stab} for $m\geq m_*$ with a suitable choice of $m_*$.

First, we show that $\Pi_m$ maps each $\kW_0, \dots, \kW_K$ into
itself uniformly for $h\in(0,h_*)$ and $U \in \kU$.  By
Lemma~\ref{l.RKShA} we estimate, for $W \in \kW_j$,
\begin{align}
  \norm{\Pi_m(W;U,h) - \1 U^0}{\kY_j^s}   
 &  \leq \norm{(\id - (\id - h {\a} {A} )^{-1})\1   U^0}{\kY_j^s}
  + \norm{  (\id - h {\a} {A} )^{-1}\1 \Q_m  U^0}{\kY_j^s} \notag \\
 & \quad + \Lambda \, R_* +  h \, \Lambda \, \norm{\a}{} \, M_j \,.
  \label{e.helpRK}
\end{align}
Using Lemma~\ref{l.RKShA}, we can find $h_*(R, U^0)$ and $m_*(R, U^0)$
such that for $m\geq m_*$ and $h \in (0,h_*)$ the first line on the
right-hand side of \eqref{e.helpRK} is less than $R/4$.  By possibly shrinking
$h_*$ further, the second line is less than $3R/4$, so that $\Pi_m (\,
\cdot \,; U, h)$ maps $\kW_j$ into itself.  Second, assumptions (B1)
and Lemma~\ref{l.RKShA} ensure that $\Pi_m$ satisfies condition (i) of
Theorem \ref{t.cm-stab} for $m\geq m_*$.  The contraction estimate,
condition (ii) of Theorem \ref{t.cm-stab}, follows from
\begin{align}
  \norm{\D_W \Pi_m(W;U,h)}{\kY_j^s \to \kY_j^s} 
  & \leq h \, \norm{(\id - h {\a} {A})^{-1}\a}{\kY_j^s \to \kY_j^s} \,
         \norm{\D B_m(W)}{\kY_j^s \to \kY_j^s} 
    \notag \\
  & \leq h \, \Lambda \, \norm{\a}{} \, M'_j
    \label{e.contraction}
\end{align}
for $j=0,\ldots, K$.  Thus, by possibly shrinking $h_*$ again, the
right-hand side bound can be made less than $1$, and condition (ii) is met
for $m\geq m_*$.  Thus Theorem \ref{t.cm-stab} (a) implies that
$W_m^j$ is of class \eqref{e.T-diff-gen} for $j=1,\ldots, s$.  The
same holds true for $\Psi_m$ due to the chain rule on scales of Banach
spaces, Lemma~\ref{l.diff} (a), applied to \eqref{eq:NewRKUpdate}
using Lemma~\ref{l.RKShA}.
\end{proof}

As for the semiflow $\Phi_m$, there is also a uniform version of this
result.

\begin{theorem}[Regularity of Galerkin truncated time-discretization,
uniform version]
\label{t.h-diff-unif} 
Assume \textup{(A)} and \textup{(B1)}, as well as \textup{(RK1)} and
\textup{(RK2)}.  Pick $\delta \in (0,\delta_*]$ such that
$\kD^{-\delta}_{K+1}$ is nonempty.  Then there is $h_* =
h_*(\delta)>0$ and $m_* = m_*(\delta)$ such that for $m \geq m_*$ the
statements of Theorem \ref{t.h-diff} hold true with $R=\delta$ and
bounds which are uniform in $U^0 \in \kD^{-\delta}_{K+1}$ and $m\geq
m_*$.  Moreover, for $m\geq m_*$, the components $W_m^j$ of the stage
vector $W_m(U,h)$ are of class \eqref{e.T-diff-gen-unif} with $T_*$
replaced by $h_*$, and, if the Runge--Kutta matrix $\a$ is invertible,
the numerical time-$h$ map $\Psi_m$ is also of class
\eqref{e.T-diff-gen-unif} with $T_*$ replaced by $h_*$.  The bounds on
$W_m$, $\Psi_m$ and $h_*$ are independent of $m\geq m_*$ and only
depend on the bounds afforded by \textup{(B1)}, \eqref{e.Rk}, and
\eqref{e.RKShA.1}, on the coefficients of the method, and on $\delta$.
\end{theorem}
 
\begin{proof}
Let $\kZ_j=\kY_j$, $\kW_j = \kB_R^{\kY_k^s}(\1 U^0)$, and $\kU =
\kB_{R_*}^{\kY_K}(U^0)$ for $j = 0, \dots, K$ as in the proof of
Theorem~\ref{t.h-diff}, taking $R=\delta$ and $R_*$ as in
Theorem~\ref{t.h-diff}.  First, due to (B1), the map $\Pi_m$ is
well-defined from $\kW_j\times\kU\times\kI$ into $\kZ_j$ with the
required regularity properties and with bounds that are uniform in $m$
and $U^0 \in \kD_{K+1}^{-\delta}$.  To show that $\Pi_m$ maps $\kW_j$
back into $\kW_j$ for $m\geq m_*(\delta)$ with a suitable choice of
$m_*$, note that, for $j=0,\ldots, K$,
\begin{align}
  & \norm{(\id - (\id - h {\a} {A} )^{-1} \1 )  U^0}{\kY^s_j}
  + \norm{(\id - h {\a} {A} )^{-1} \1  \Q_m U^0}{\kY^s_j}
  \notag \\ 
  & \qquad 
    \leq h \, \max_{s \in [0,h]} 
           \norm{\a A (\id - s {\a} {A} )^{-2} \1 U^0}{\kY^s_j}
         + \Lambda \, \norm{\1\Q_m U^0}{\kY^s_j}
  \notag \\
  & \qquad 
    \leq (h \, \Lambda^2 \, \norm{\a}{} + \Lambda/m)  \, 
         \sup_{U \in \kD_{K+1}^{-\delta}} \norm{U}{\kY_j} 
  \notag \\
  & \qquad 
    \leq (h \, \Lambda^2 \, \norm{\a}{} + \Lambda/m) \, R_{K+1} \,, 
  \label{e.uniformh}
\end{align}
where $R_{K+1}$ is defined by \eqref{e.Rk}.  Inserting this estimate
into \eqref{e.helpRK}, we see that we can choose $h_*(\delta)>0$ small
enough and $m_*(\delta)$ big enough such that $\Pi_m( \, \cdot \,;
U,h)$ maps each $\kW_j$ into itself and, due to \eqref{e.contraction},
such that $\Pi_m$ is a contraction on each $\kW_j$ uniformly for $U^0,
U \in \kD_{K+1}^{-\delta}$, and $h \in [0,h_*]$.  So the conditions of
Theorem~\ref{t.cm-stab} are satisfied uniformly for $m\geq m_*$, $h\in
(0,h_*)$, and $U^0 \in \kD_{K+1}^{-\delta}$.

Applying $A$ to $\Pi_m(W,U,h)$ yields
\[
  AW_m = A \Pi_m(W_m,\P_m U,h)= (\id - h  {\a} {A} )^{-1} \1 A \P_m U + 
  h \a A  (\1 U - h  {\a} A)^{-1} \P_m {B}(W_m) \,.
\]
Using Lemma~\ref{l.diff} (a), the chain rule on a scale of Banach
spaces, together with the estimates of Lemma~\ref{l.RKShA}, we find
that $A W_m^j$ is of class \eqref{e.Aphim-unif} for $j=1,\ldots, s$.
Moreover, on the $(K+1)$-scale $\{\kZ_j\}$, given by
 $\kZ_j=\kY_j$ for $j=0,\ldots, K$ and
$\kZ_{K+1}=\kY_K$, setting  $\kW_j =\kD_j^s$ for $j=0,\ldots K$,
$\kW_{K+1}=\kD_K^s$,    $\kU=\kD_{K+1}^{-\delta}$,
$\kX=\kY_{K+1}$ and $\kI=(0,h_*)$, the map $\Pi_m$ satisfies
conditions (i) and (ii) of Theorem~\ref{t.cm-stab}.  (Here we have
used Lemma~\ref{l.RKShA}, in particular \eqref{e.RKShA.B2}, once
again.)  Therefore, $\partial_h W_m^j$ is of class
\eqref{e.Aphim-unif}.  Then Lemma~\ref{l.technical1} implies that
$W_m^j$ is of class \eqref{e.T-diff-gen-unif}.  If $\a$ is invertible,
we can use \eqref{eq:NewRKUpdate}, \eqref{e.RKShA.B} and the chain
rule in the form of Lemma \ref{l.diff} (a) to show that $\Psi_m$ is
also of class \eqref{e.T-diff-gen-unif}.
\end{proof}

\begin{remark} \label{r.discr-improved-t-diff}
When $\a$ is not invertible, an appropriate modification of the proof of
Theorem~\ref{t.h-diff-unif} yields the weaker statement
\begin{equation}
  \label{e.Psi-improved-t-diff-local}
    \Psi_m \in \bigcap_{\substack{  j+k\leq N-1 \\ k \leq K }}
  \kCb^{(\underline j, k+1)} (\kD_{K+1}^{-\delta} \times [0,T_*]; 
  \kD)
\end{equation} 
for $m\geq m_*$ with bounds that depend only on the bounds afforded by
\textup{(B1)}, \eqref{e.Rk}, and \eqref{e.RKShA.1}, on the
coefficients of the method, and on $\delta$.  An analogous statement
holds true for $\Psi$ \cite[Remark~3.22]{OW2010}.
\end{remark}


\subsection{Convergence of Galerkin truncated time discretization}

Next, we prove a convergence result for the time-semidiscretization
of the projected system.  The error bounds are uniform in the spatial
truncation parameter $m$.
 
\begin{theorem}[Convergence of time discretization of projected system]  
\label{c.RKConv-m}
Assume that the semilinear evolution equation \eqref{e.pde} satisfies
condition \textup{(A)}, and apply a Runge--Kutta method of classical
order $p$ subject to conditions \textup{(RK1)} and \textup{(RK2)} to
it.  Assume further that \textup{(B1)} holds with $K \geq p$.  Pick
$\delta \in (0,\delta_*]$ such that $\kD_{p+1}^{-\delta}$ is
non-empty, and fix $T>0$.  Then there exist positive constants $h_*$,
$m_*$, $c_1$, and $c_2$ that only depend on the bounds afforded by
\textup{(B1)}, \eqref{e.RKShA.1}, on the coefficients of the method,
and on $\delta$, such that for every $U^0$ satisfying
\begin{equation}
  \label{e.unifConvAss}
  \{ \Phi^t(U^0) \colon t \in [0,T] \} \subset \kD^{-\delta}_{p+1} \,,
\end{equation}
$h \in [0,h_*]$, and $m \geq m_*$, the numerical solution
$(\Psi_m^h)^n (U^0)$ lies in $\kD$ and satisfies
\begin{equation}
  \label{e.m-unifConv}
  \norm{(\Psi_m^h)^n(U^0) - \Phi_m^{nh}(U^0)}{\kY} 
  \leq c_2 \, \e^{c_1 nh} \, h^p
\end{equation}
so long as $nh \leq T$.  
\end{theorem}

\begin{proof}
Convergence of the time semidiscretization under the above assumptions
can be proved by a standard Gronwall argument, see \cite[Theorem
3.24]{OW2010},  condition (A1) of which is always satisfied in the
setting here (it is stated as \eqref{e.RKShA.c} in
Lemma~\ref{l.RKShA}).  Then \cite[Theorem 3.24]{OW2010} asserts that
whenever a semiflow $\Phi$ satisfies \eqref{e.unifConvAss}, there
exist constants $c_1$ and $c_2$ that only depend on the bounds
afforded by \textup{(B1)} and \eqref{e.RKShA.1}, on the coefficients of
the method, and on $\kD_{p+1}^{-\delta}$ such that
\begin{equation}
  \label{e.unifConv}
  \norm{(\Psi^h)^n(U^0) - \Phi^{nh}(U^0)}{\kY} 
  \leq c_2 \, \e^{c_1 nh} \, h^p
\end{equation}
so long as $nh \leq T$. 

Here, we need to apply this result with $\Phi$ replaced by $\Phi_m$
and $\Psi$ by $\Psi_m$ and show that we obtain uniform bounds in
$m\geq m_*$ and $U^0$ satisfying \eqref{e.unifConvAss}.  So we have to
show that if condition \eqref{e.unifConvAss} holds, then there is an
analogue of this condition for the truncated system which is valid for
all $m\geq m_*$ In other words, we have to find a $\kY_{p+1}$-bounded
set $ \tilde\kD_{p+1} \subset \kD_p \cap \kY_{p+1}$
 and $\tilde\delta>0$ such that
\begin{equation}
  \label{e.orbit-condition}
  \{ \Phi_m^t(U^0) \colon t \in [0,T] \} 
  \subset \tilde\kD_{p+1}^{-\tilde\delta}
\end{equation}
holds for all $U^0$ satisfying \eqref{e.unifConvAss} and all $m\geq
m_*$.  Applying Corollary~\ref{c.wpp.unif} with $\kY_p$ in place of
$\kY$ and $\kD_p$ in place of $\kD$, we find that there is some $m_*
\in \N$ such that
\[
  \sup_{t \in [0,T]} 
  \norm{\Phi^t(U^0) - \Phi^t_m(U^0)}{\kY_{p}} < \delta/2
\]
and 
\[
  \sup_{t \in [0,T]} 
  \norm{\Phi^t_m(U^0)}{\kY_{p+1}} \leq C
\]
for some $C>0$, all $m \geq m_*$, and all $U^0$ satisfying
\eqref{e.unifConvAss}.  Thus, with $\tilde\delta = \delta/2$ and
\[
  \tilde\kD_{p+1} 
  = \kD_p \cap \interior \kB_{C+\delta}^{\kY_{p+1}}(0) \,,
\]   
where $\interior(\kU)$ denotes the interior of a set $\kU$ of a Banach
space $\kX$, condition \eqref{e.orbit-condition} is satisfied for all
$m\geq m_*$.  This completes the proof.
\end{proof}

By combining this theorem with Theorem~\ref{c.wpp.unif} we obtain
convergence of the space-time discretization to the semiflow
$\Phi^t(U^0)$ of order $O(h^p) + O(m^{-K-1})$ for $t\in [0,T]$ and
$m\geq m_*$ with uniform bounds for all $U^0$ satisfying
\eqref{e.unifConvAss}.  In particular, we do not require a coupling
between spatial resolution $m$ and temporal resolution $h$ for this
convergence result.


\subsection{Accuracy of derivatives of Galerkin truncated time
discretization}
\label{ssec:derivTruncate}

Results corresponding to Theorems~\ref{t.h-diff}
and~\ref{t.flow-projection-error-scale} hold true for the stability
under spectral truncation of the numerical stage vector and its
derivatives.

\begin{theorem}[Projection error for the numerical method, local version]
\label{t.projection-error-scale} 
Assume \textup{(A)}, \textup{(B1)}, \textup{(RK1)}, and
\textup{(RK2)}.  Fix $R\in (0,\delta_*]$ such that $\kD_K^{-R}$ is
nonempty and choose $U^0 \in \kD_K^{-R}$.  Let $h_*=h_*(R,U^0)>0$ and
$m_*=m_*(R,U^0)$ be as in Theorem~\ref{t.h-diff}.  Then for every $0
\leq P \leq K$,
\begin{subequations}
  \label{e.m-s}
\begin{gather}
  \norm{W - W_m }{N-1-P,K-P} = O(m^{-P}) 
  \label{e.w-m-s}
\intertext{and}
  \norm{\Psi^h - \Psi^h_m }{N-1-P,K-P} = O(m^{-P}) \,,
  \label{e.psi-m-s}
\end{gather}
\end{subequations}
where the norm in \eqref{e.m-s} is defined with respect to the spaces
\eqref{e.localWSpaces}.  The order constants depend only on the bounds
afforded by \textup{(B1)} and \eqref{e.RKShA.1}, on the coefficients
of the method, and on $R$.
\end{theorem}

\begin{proof}
The proof of \eqref{e.m-s} is an application of
Theorem~\ref{t.cm-stab} on the stability of contraction mappings
where, as in the setting of Theorem~\ref{t.h-diff}, $\Pi$ is defined
by \eqref{eq:newRKStagesIter}, $\Pi_m= \P_m \circ \Pi$, and we set
\begin{equation}
  \label{e.localWjSpace}
  \kW_j = \kB_{R}^{\kZ_j} (\1 U^0)  
  \qquad \text{where} \qquad \kZ_j = \kY_j^s 
\end{equation}
for $j = 0, \dots, K$.  We further identify $\kI = (0, h_*)$, $w = W$,
$\mu=h$, and $\kX=\kY_K$ with $\kU \equiv \kU_K =
\kB_{R_*}^{\kY_K}(U^0)$.

We already verified in the proof of Theorem~\ref{t.h-diff} that
conditions (i) and (ii) of Theorem~\ref{t.cm-stab} hold with uniform
bounds for $m\geq m_*$.  Thus, Theorem~\ref{t.cm-stab} (b) yields a
bound of the form \eqref{e.w-m-s} provided we can show that
\begin{equation}
  \label{e.PiPim-h}
  \Norm{\Pi-\Pi_m}{N-1,K,0,P} = O(m^{-P}) \,.
\end{equation}
Writing
\begin{multline}
  \Pi(W; U, h) - \Pi_m(W; U, h) 
  = (\id - h \a A)^{-1} \, \1 \Q_m U  
    + h \a \, (\id - h \a A)^{-1} \, \Q_m B(W) \,,
  \label{e.splitting}
\end{multline}
we apply Lemma~\ref{l.diff} to both terms on the right-hand side as follows.
For the first term,
\[
  G_m(W; U, h) \equiv \Q_m(\id - h \a A)^{-1} \1 \, \Q_m U \,,
\]
we take $\Pi=\id$, $\Sigma=\P_m$, and $v(U,h)=w(U,h) = (\id-h \a
A)^{-1} \1 U$ to conclude that there exists $c_1$ such that
\begin{align}
  \Norm{G_m}{N-1,K,0,P} 
  = \norm{G_m}{N-P-1,K-P} 
  \leq c_1 \Norm{\Q_m}{N-1,K,0,P}
  = O(m^{-P}) \,,
  \label{e.1stTerm_Tech}
\end{align}
where the first equality is due to \eqref{e.3Norm2Norm}, we recall
Lemma~\ref{l.RKShA} for the differentiability properties of $(\id - h
\a A)^{-1}$, and note that the final statement is due to estimate
\eqref{e.qm-estimate} on the Galerkin remainder.

To estimate the second term on the right-hand side of \eqref{e.splitting},  
apply Lemma~\ref{l.diff} for each $\kappa \in P, \dots, K$ with $\Pi =
\id$, $\Sigma=\P_m$, $u=W$, and $v(u,h)=w(u,h)= h \a \, (\id - h \a
A)^{-1} B(W)$, $\kU$ replaced by $\kU_\kappa \equiv
\kB_{R}^{\kZ_\kappa}(\1 U^0)$, $N$ replaced by $N-1$, and $K$ replaced
by $\kappa$.  Hence, by Lemma~\ref{l.RKShA}, there is some $c_2$ such
that
\[
  \norm{\Q_m \, h \a \, (\id - h \a A)^{-1} B(W)}{N-1-P,\kappa-P}
  \leq c_2 \, \Norm{\Q_m}{N-1,\kappa,0, P} = O(m^{-P}) \,.
\]
Then, Lemma~\ref{l.technical3} with $N$ replaced by $N-1$ and $S=P$
implies
\begin{equation}
  \Norm{\Q_m \, h \a \, (\id - h \a A)^{-1} B}{N-1,K,0,P}
  = O(m^{-P}) \,.
  \label{e.2ndTerm_Tech}
\end{equation}
The constants $c_1$ and $c_2$ depend only on the bounds from (B1) and
the bounds from Lemma~\ref{l.RKShA}.  Altogether, we have proved
\eqref{e.PiPim-h}; the proof of \eqref{e.w-m-s} is complete.

To prove estimate \eqref{e.psi-m-s}, note that by
\eqref{eq:NewRKUpdate}, 
\begin{equation}
  \label{e.Dhpsi-psim}
  \Psi^h(U) - \Psi^h_m(U)
  = \mS(hA) \Q_m U 
      + {\b}^T \, (J(W(U,h),h) - \P_m J(W_m(U,h),h)) \,,
\end{equation}
where
\begin{equation}
  \label{e.JW}
  J(W;U,h) = h (\id - h \a A)^{-1}B(W ) \,,
\end{equation}
so that
\begin{align}
  \norm{\Psi^h - \Psi^h_m }{N-1-P,K-P} 
  & \leq \norm{\mS(h A) \Q_m U}{N-1-P,K-P}
    \notag \\ 
  & \quad + s \, \norm{\b}{} \, 
      \norm{J \circ W-\P_m J \circ W_m}{N-1-P,K-P} \,.
  \label{e.Psipsim}
\end{align}
We estimate the first term of \eqref{e.Psipsim} using
Lemma~\ref{l.diff} with $\Pi=\id$, $\Sigma=\P_m$, $u=U$,
$\kZ_j=\kY_j$, $\kW_j = \kB_{R}^{\kY_j}(U^0)$ for $j=0, \dots, K$, and
$w(u,h) = v(u,h)=\mS(h A) U$.  For the second term of
\eqref{e.Psipsim}, we use Lemma~\ref{l.diff} with $\Pi=J$, $\Sigma =
\P_m J$, $\kZ_j$ and $\kW_j$ from \eqref{e.localWjSpace} as before,
$w(U,h) = W(U,h)$, and $v(U,h) = W_m(U,h)$.  Thus, by
Lemma~\ref{l.RKShA}, there exists $c_4$ such that
\begin{align}
  \norm{\Psi^h - \Psi^h_m}{N-1-P,K-P} 
  & \leq c_4 \, \Norm{\Q_m U}{N-1,K,0,P}+ 
         c_4 \, \Norm{\Q_m J}{N-1,K,0,P}
    \notag \\
  & \quad + c_4 \, \norm{W-W_m}{N-1-P,K-P} \,.
    \label{e.PsiPsihmDeriv}
\end{align}
The first term is $O(m^{-P})$ by \eqref{e.qm-estimate}; to obtain the
required estimate for the second term we proceed as in the computation
proving \eqref{e.2ndTerm_Tech}, but with $h (\id - h \a A)^{-1} B$ in
place of $h \a (\id - h \a A)^{-1} B$; the third term is $O(m^{-P})$
by \eqref{e.w-m-s}.
\end{proof}

\begin{theorem}[Projection error for the numerical method, uniform
version] 
\label{t.projection-error-scale-unif} 
Assume \textup{(A)}, \textup{(B1)}, \textup{(RK1)}, and
\textup{(RK2)}.  Choose $\delta \in (0,\delta_*]$ small enough such
that $\kD_{K+1}^{-\delta}$ is nonempty, and let $h_* = h_*(\delta)>0$
and $m_*=m_*(\delta)$ be as in Theorem~\ref{t.h-diff-unif}.  Then
\eqref{e.m-s} holds true with respect to the uniform setting
\eqref{e.uniformWSpaces}.  Moreover, for every $0 \leq P \leq K+1$ and
$N>K+1$,
\begin{subequations}
  \label{e.m-s-unif}
\begin{gather}
  \norm{W - W_m}{N-1-P,K+1-P} = O(m^{-P}) 
  \label{e.w-m-s-unif}
\intertext{and, for $\a$ invertible,}
  \norm{\Psi^h - \Psi^h_m}{N-1-P,K+1-P} = O(m^{-P}) \,,
  \label{e.psi-m-s-unif}
\end{gather}
\end{subequations}
where the norm in \eqref{e.m-s-unif} is defined with respect to the
spaces \eqref{e.uniformWSpaces}.  The order constants depend only on
the bounds afforded by \textup{(B1)}, \eqref{e.Rk}, and
\eqref{e.RKShA.1}, on the coefficients of the method, and on $\delta$.
\end{theorem}
 
\begin{proof}
For each $U^0 \in \kD_{K+1}^{-\delta}$, define $\Pi$, $\Pi_m$,
$\kW_j=\kB_{\tilde{R}}^{\kY_j^s}(\1 U^0)$, $\kZ_j=\kY_j^s$,
$\kI=(0,h_*)$, etc., as in the proof of Theorem~\ref{t.h-diff}.  We
first show that \eqref{e.m-s} holds true with respect to the uniform
setting \eqref{e.uniformWSpaces}.  We already verified in the proof of
Theorem~\ref{t.h-diff-unif} that conditions (i) and (ii) of
Theorem~\ref{t.cm-stab} hold uniformly in $U^0 \in
\kD_{K+1}^{-\delta}$ and $m\geq m_*$.  Further, by checking uniformity
of all required estimates, we verify that \eqref{e.PiPim-h}, in the
proof of Theorem~\ref{t.projection-error-scale}, holds uniformly in $U^0
\in \kD_{K+1}^{-\delta}$.  Applying Theorem~\ref{t.cm-stab} (b) for
all $U^0 \in \kD_{K+1}^{-\delta}$ then implies that \eqref{e.w-m-s}
holds in the uniform setting \eqref{e.uniformWSpaces}.  Similarly,
\eqref{e.PsiPsihmDeriv} holds uniformly in $U^0 \in
\kD_{K+1}^{-\delta}$ so that \eqref{e.psi-m-s} holds with respect to
the uniform setting \eqref{e.uniformWSpaces}.
  
We next show that, for $N>K$,
\begin{subequations}
\label{e.m-s-A}
\begin{equation}
  \norm{AW - AW_m}{N-1-P,K-P} = O(m^{-P}) 
  \label{e.w-m-s-A}
\end{equation}
and, for $\a$ invertible,
\begin{equation}
  \norm{A\Psi^h - A\Psi^h_m }{N-1-P,K-P} = O(m^{-P}) \,,
  \label{e.psi-m-s-A}
\end{equation}
\end{subequations}
in the uniform setting \eqref{e.uniformWSpaces}.  We apply $A$ onto
\eqref{e.splitting} as well as onto the expression for
$\Psi^h-\Psi^h_m$.  The resulting difference expressions are then
estimated as in the proof of Theorem~\ref{t.projection-error-scale}
using Lemma \ref{l.RKShA}, in particular \eqref{e.RKShA.b}.  Now, we
aim to show that, for $N>K+1$ and $0 \leq P \leq K$,
\begin{subequations}
  \label{e.dtWAndPsi-WmAndPsim1}
\begin{equation}
  \label{e.dtW-wm}
  \norm{\partial_hW - \partial_h W_m}{N-P-2,K-P} = O(m^{-P})
\end{equation}
and 
\begin{equation}
  \label{e.dtPsi-psim}
  \norm{\partial_h\Psi - \partial_h \Psi_m}{N-P-2,K-P} = O(m^{-P}) \,.
\end{equation}
\end{subequations}
To prove \eqref{e.dtW-wm}, we apply Theorem~\ref{t.cm-stab} on the
stability of contraction mappings to the pair $\Pi$ from
\eqref{eq:newRKStagesIter} and $\Pi_m = \P_m \circ \Pi$, but this time
on the $(K+1)$ scale $\kZ_j =\kY^s_j$ for $j=0, \dots, K$ and
$\kZ_{K+1} = \kY^s_K$ with $\kW_j = \kD^s_j$ for $j=0, \dots, K$ and
$\kW_{K+1} = \kD_K^s$.  Set, as before, $\kU=\kD_{K+1}^{-\delta}$,
$\kX = \kY_{K+1}$, and $\kI=(0,h_*)$.  Due to \eqref{e.RKShA.B2} and
(B1), the map $\Pi_m$ satisfies the assumptions of
Theorem~\ref{t.cm-stab} in this setting for $m\geq m_*$.  We obtain
$\Norm{\Pi-\Pi_m}{N-1,K+1,0,P} = O(m^{-P})$ as in the proof of
Theorem~\ref{t.projection-error-scale}. By Theorem~\ref{t.cm-stab}
(b), this implies $\norm{W-W_m}{N-P-1,K+1-P} = O(m^{-P})$ with respect
to the above defined hierarchy, and in particular \eqref{e.dtW-wm}.
   
Estimate \eqref{e.dtPsi-psim} is proved similarly.  We estimate the
norms of both terms in \eqref{e.Dhpsi-psim}, with $J$ as in
\eqref{e.JW}.  First, as in the proof of Theorem
\ref{t.projection-error-scale}, using Lemma~\ref{l.diff} with
\[
 \Pi=\id, \Sigma=\P_m, u=U, \kZ_j=\kY_j, \kW_j =
\kB_{r}^{\kY_j}(U^0) ~~\mbox{for}~~j=0, \dots, K+1,
\]
 where $r>0$ is such that
$\kD_{K+1} \subset \kB_r^{\kY_{K+1}}(0)$, and $w(u,h) = v(u,h)=\mS(h
A) U$, we obtain
\[
  \norm{\mS(hA) \Q_m U}{N-P-1,K+1-P}= O(m^{-P}).
\]
Thus, in particular, 
\begin{equation}
  \label{e.DhSmhA}
  \norm{\partial_h \mS(hA) \Q_m U}{N-P-2,K-P} = O(m^{-P})
\end{equation}
holds with respect to the uniform setting \eqref{e.uniformPsiSpaces}.

We now estimate the second term of \eqref{e.Dhpsi-psim}.  Consider the
$(K+1)$-scale from above, i.e., $\kZ_j =\kY^s_j$ for $j=0, \dots, K$,
$\kZ_{K+1} = \kY^s_K$, $\kW_j = \kD^s_j$ for $j=0, \dots, K$,
$\kW_{K+1} = \kD_K^s$, $\kU=\kD_{K+1}^{-\delta}$, $\kX = \kY_{K+1}$,
and $\kI=(0,h_*)$.  Due to \eqref{e.RKShA.B2} and (B1), the maps $J$
and $\P_m J$ satisfy the assumptions of Lemma~\ref{l.diff} in this
setting, and, as we have seen above, $W$ and $W_m$ also satisfy the
conditions of Lemma~\ref{l.diff} for this choice of scale and $m\geq
m_*$.  We obtain $\Norm{Q_m J}{N-1,K+1,0,P} = O(m^{-P})$ as in the
proof of Theorem~\ref{t.projection-error-scale}, and
\[
 \norm{J \circ W - \P_m J \circ W_m}{N-1-P,K+1-P} = O(m^{-P})
\]
for the above choice of scale.  This implies that, with respect to the
uniform setting \eqref{e.uniformPsiSpaces}, we have
\[
  \norm{\partial_h (J \circ W - \P_m J \circ W_m)}{N-2-P,K-P}
  = O(m^{-P})
\]
which, together with \eqref{e.DhSmhA} and \eqref{e.Dhpsi-psim},
implies \eqref{e.dtPsi-psim}.

Finally, Lemma~\ref{l.technical1}, given that \eqref{e.m-s} holds
in the uniform setting \eqref{e.uniformPsiSpaces} and together with
estimates \eqref{e.m-s-A} and \eqref{e.dtWAndPsi-WmAndPsim1}, implies
\eqref{e.m-s-unif}.
\end{proof}
 
\begin{remark} 
If, in the setting of Theorem~\ref{t.projection-error-scale-unif}, the
matrix $\a$ is not assumed to be invertible, then \eqref{e.dtPsi-psim}
still holds, cf.\ Remark~\ref{r.discr-improved-t-diff}.
\end{remark}

\appendix

\section{Stability of contraction mappings}
\label{s.appendix}

Abstract contraction mapping theorems on a scale of Banach spaces have
been obtained in \cite{OW2010,Vand87,Wulff00}.  For the results in
this paper, we must, in addition, estimate the stability of the fixed
point under perturbation of the contraction map.
 
For $K \in \N_0$, let ${\kZ} = {\kZ}_0 \supset {\kZ}_1 \supset\ldots
\supset {\kZ}_{K}$ be a scale of Banach spaces, each continuously
embedded in its predecessor, and let $\kV_j, \kW_j \subset \kZ_j$ be
nested sequences of sets.  Let $\kX$ be a Banach space, and let $\kU
\subset \kX$ and $\kI \subset \R$ be open.  We note that all results
in this section easily extend to the case where $\kI$ is an open
subset of $\R^p$.  We may assume that $\lVert w \rVert_{\kZ_j} \leq
\lVert w \rVert_{\kZ_{j+1}}$ for all $w \in \kZ_{j+1}$.  (If this is
not the case, we inductively equip $\kZ_{j+1}$ with the equivalent
norm $\lVert \, \cdot \, \rVert_{\kZ_{j+1}} + \lVert \, \cdot \,
\rVert_{\kZ_{j}}$.)

As detailed in Section~\ref{ss.notation}, we use the following
additional integer indices.  The minimal regularity we guarantee for
the image space of the function considered is the scale rung $L$, the
``loss index'' $S$ indicates how many rungs on the scale the range of
a function is down relative to its domain, and $N$ denotes the maximal
regularity of the function.  We assume $0 \leq L \leq K-S \leq N-S$.
We work with the
family of spaces
\begin{subequations}
\begin{equation}
  \label{e.CNKLS}
  \kC_{N,K,L,S}(\{\kV_j\}, \kU,\kI; \{ \kW_j\}) =
  \bigcap_{\substack{i+j+k \leq N-S \\ 
                     L + \ell \leq k \leq K-S}}
  \kCb^{(\underline i, \underline j, \ell)}
    ({\kV}_{k+S} \times{ \kU} \times  \kI; \kW_{k-\ell}) \,,
\end{equation} 
endowed with the  norm \eqref{e.norm4Pi}, and abbreviate
\begin{gather}
  \label{e.CNKS}
  \kC_{N,K,L}(\{\kV_j\}, \kU,\kI; \{ \kW_j\}) 
  = \kC_{N,K,L,0}(\{\kV_j\}, \kU,\kI; \{ \kW_j\}) \,, \\
  \label{e.CNK00}
  \kC_{N,K}(\{\kV_j\}, \kU,\kI; \{ \kW_j \}) 
  = \kC_{N,K,0,0}(\{\kV_j\}, \kU,\kI; \{ \kW_j\})
\end{gather} 
with corresponding norms \eqref{e.norm3Pi} and \eqref{e.norm2Pi},
respectively.  We note that any function of class \eqref{e.CNKLS} has
a maximal number of $N-L-S$ derivatives in its first and second
argument on the lowest admissible domain scale $\kZ_{L+S}$.

Furthermore, let
\begin{equation}
  \label{e.CNKL}
  \kC_{N,K,L}( \kU,\kI; \{ \kW_j \})
  = \bigcap_{\substack{j+k \leq N \\ L+\ell \leq k \leq K}}
      \kCb^{(\underline j,\ell)} (\kU \times \kI; \kW_{k-\ell}) \,,
\end{equation}
endowed with the norm \eqref{e.norm3w}, and abbreviate
\begin{equation}
  \label{e.CNK0}
  \kC_{N,K}( \kU,\kI; \{\kW_j\}) = \kC_{N,K,0}( \kU,\kI; \{\kW_j\})
\end{equation}
\end{subequations}
with corresponding norm \eqref{e.norm2w}.  For future reference, we
note the following.

\begin{remark} \label{r.no-w} 
When a map $\Pi \in \kC_{N,K,L,S}(\{\kV_j\}, \kU,\kI; \{ \kW_j\})$
does not depend on $w$, it can be interpreted as an element from
$\kC_{N,K,L}(\kU,\kI; \{ \kW_j\})$ where, by \eqref{e.3Norm2Norm},
$\Norm{\Pi}{N,K,L,S} = \norm{\Pi}{N-S,K-S,L}$.
\end{remark}

We simply write $\kC_{N,K,L,S}$ and $\kC_{N,K,L}$ when the arguments
are unambiguous.  We also write
\[
  \partial_\mu \Pi(w(u,\mu);u,\mu) 
  = \partial_\mu \Pi(w;u,\mu) \big|_{w=w(u,\mu)}
  = (\partial_\mu \Pi \circ w)(u,\mu)
\] 
to denote partial $\mu$-derivatives whereas we write \ $\D_\mu
(\Pi(w(u,\mu),u,\mu))$ to denote full $\mu$-derivatives.

We begin with four technical lemmas which can be proved by simple
index arithmetic.  Details can be found in \cite{OW2010}.

\begin{lemma} \label{l.cm-scale}
If $N>K$  then, with $\kW \equiv \kW_0$,
\[
  \kC_{N,K}(\kU,\kI;\{\kW_j\}) \subset \kCb^K (\kU \times \kI; \kW) \,.
\]
\end{lemma}

\begin{lemma} \label{l.technical2} 
Suppose that 
\begin{itemize}
 \item[(i)] $w \in \kC_{N,K,L}(\kU,\kI;\{\kW_j\})$;
\item[(ii)]  the map
$(u,\tilde{u},\mu) \mapsto \D_u w(u,\mu) \tilde{u}$ is of class
$\kC_{N,K,L}(\kU \times \kB_1^\kX(0),\kI; \{\kZ_j\})$.
\end{itemize}
   Then $w \in
\kC_{N+1,K,L}(\kU,\kI;\{\kW_j\})$ and
\[
  \norm{w}{N+1,K,L} 
  \leq \sup_{\|\tilde{u}\|_{\kX} \leq 1}
  \norm{\D_u w \, \tilde{u}}{N,K,L} + \norm{w}{N,K,L} \,.
\]
\end{lemma}
 
\begin{lemma} \label{l.technical1}
When $N>K$, $w \in \kC_{N,K+1,L+1}(\kU,\kI;\{\kW_j\})\cap
\kC_{N,L,L}(\kU,\kI;\{\kW_j\}) $, and $\partial_\mu w \in
\kC_{N-1,K,L}(\kU,\kI;\{\kZ_j\})$, then
\[
  w \in \kC_{N,K+1,L} (\kU,\kI;\{\kW_j\})
\]
and
\[
  \norm{w}{N,K+1,L} 
  \leq \norm{w}{N,K+1,L+1} 
       + \norm{w}{N,L,L}
       + \norm{\partial_\mu w}{N-1,K,L} \,.
\]
\end{lemma}
 
\begin{lemma} \label{l.technical3}
We have
\[
  \bigcap_{S \leq \kappa\leq K} \kC_{N-S,\kappa-S,L}(\kV_\kappa \times
    \kU, \kI; \{ \kW_j\})
  = \kC_{N,K,L,S}(\{\kV_j\},\kU,\kI; \{\kW_j\}) \,,
\]
and
\[
  \Norm{\Pi}{\kC_{N,K,L,S}(\{\kV_j\},\kU;\kI;\{ \kW_j \})} 
  \sim \max_{S \leq \kappa \leq K}
       \norm{\Pi}{\kC_{N-S,\kappa-S,L}(\kV_\kappa \times \kU,\kI;
                  \{ \kW_j \})} \,,
\]
where $\sim$ denotes that left hand and right-hand sides provide
equivalent norms on $\kC_{N,K,L,S}$.
\end{lemma}

We now prove a stability result for fixed points of contraction
mappings, i.e., we want to bound norms of differences of fixed points
in terms of norms of differences of contraction maps.  To do so, we
first need to look at a corresponding stability result for
compositions.

The following lemma states that the difference between two functions
which are both compositions of functions can be estimated by the
difference of the outer functions and the difference of the inner
functions, and that the same holds for derivatives of the difference.
Here $S$ is the minimal smoothness of the image of the inner functions
and of the domain of the outer functions, $K$ is the number of scales
and $N-S$ is the maximal number of derivatives of the inner functions.
Finally, $P$ is used to relax the required smoothness in the
estimates.

\begin{lemma}[Stability of compositions] \label{l.diff} 
For $0 \leq S+P \leq K \leq N$, suppose $\Pi=\Pi(w; u,\mu)$,
$\Sigma=\Sigma(w; u,\mu)$, $w=w(u,\mu)$, and $v=v(u,\mu)$ satisfy
\begin{itemize}
\item[(i)] $\Pi, \Sigma \in \kC_{N+1,K,0,S}( \{ \kW_j\}, \kU, \kI; \{
  \kZ_j\})$;
\item[(ii)] $w, v \in  \kC_{N,K,S}(\kU, \kI; \{ \kW_{j} \})$.
\end{itemize}
The following then hold.
\begin{itemize}
\item[(a)] $\Pi \circ w \in \kC_{N-S,K-S}(\kU,\kI; \{ \kW_j \})$ and
$\lVert \Pi \circ w \rVert_{N-S,K-S,L}$ can be bounded by a polynomial
with non-negative coefficients in $\Norm{\Pi}{N,K,L,S}$ and $\lVert w
\rVert_{N,K,S+L}$; the same holds true for $\Sigma \circ v$.

\item[(b)] There is some $c>0$ which is a polynomial with non-negative
coefficients in $\Norm{\Pi}{N+1,K,0,S}$, $\tnorm{w}{N,K,S}$,
$\Norm{\Sigma}{N+1,K,0,S}$, and $\tnorm{v}{N,K,S}$ such that
\begin{equation}
  \label{e.DiffPiCircu}
\norm{\Pi \circ w-\Sigma \circ v}{N-P-S,K-P-S} 
  \leq c \, \Norm{\Pi-\Sigma}{N,K,0,P+S} + c \, \norm{w-v}{N-P,K-P,S} \,.
\end{equation}
\end{itemize}
\end{lemma}

\begin{remark} \label{r.cm-composition} Part (a) was already shown in
\cite[Lemma A.6]{OW2010} and is the chain rule on the scale of Banach
spaces.  It will be our main tool for obtaining estimates on the scale
of Banach spaces for compositions of maps of the form $(\Pi\circ
w)(u,\mu) \equiv \Pi(w(u,\mu);u,\mu)$.  The essence of the result is
very natural: When the outer function $\Pi$ loses $S$ rungs on the
scale, the inner function $w$ must have minimal regularity $L=S$ and
the composition maps at best into the scale rung $K-S$.
\end{remark}

\begin{proof}
We prove part (b) only.  It follows the same pattern as the proof of
part (a) with the additional difficulty that we need to carefully keep
track of differences in the various spaces.  We proceed by induction
in $N$ and $K$.  For $N=K=P+S$,
\begin{align*}
  \norm{\Pi\circ w-\Sigma \circ v}{\kC(\kU \times \kI; \kZ)}
  & \leq \norm{\Pi\circ w-\Pi\circ v}{\kC(\kU \times \kI; \kZ)}
       + \norm{\Pi\circ v-\Sigma \circ v}{\kC(\kU \times \kI; \kZ)}
       \notag \\
  & \leq c_0 \, \norm{w-v}{\kC(\kU \times \kI; \kZ_{S})}
         + \Norm{\Pi - \Sigma}{P+S,P+S,0,P+S} \,,
\end{align*}
where, by the mean value theorem, $c_0 = \Norm{\Pi}{P+S+1,P+S,0,P+S}$.

Let us now increment $N$ holding $P$ and $K$ fixed.  Let $\kB \equiv
\kB_1^\kX(0)$.  By Lemma~\ref{l.technical2}, it is sufficient to
derive the claimed upper bound for the $\kC_{N-P-S,K-P-S}(\kU\times
\kB,\kI;\{\kW_j\})$ norm of the function which maps
$((u,\tilde{u}),\mu) \in (\kU\times\kB)\times \kI$ to
\begin{align}
  \D_u(\Pi\circ w-\Sigma\circ v) \, \tilde{u} 
  & = (\partial_u \Pi \, \tilde{u}) \circ w  -  
      (\partial_u \Sigma \, \tilde{u}) \circ v 
      \notag \\ 
  & \quad + \hat \Pi \circ (\D_u w \, \tilde{u}) 
          - \hat \Sigma \circ (\D_u v \, \tilde{u})  \,,
  \label{e.diff1}
\end{align}
where $\hat\Pi$ and $\hat\Sigma$ are defined in \eqref{e.hatPiSigma}
below.

To estimate the first line of the right-hand side of \eqref{e.diff1}, we define
\[
\Pi_1(w;(u, \tilde u),\mu) = \partial_u \Pi(w;u,\mu) \, \tilde{u} ~~\quad{and}~~
\Sigma_1(w;(u, \tilde u),\mu) = \partial_u \Sigma(w;u,\mu) \,
\tilde{u}.
\]  Then, by the induction hypothesis, there is a constant $c_1$
which is a polynomial in $\Norm{\Pi}{N+1,K,0,S} \geq
\Norm{\Pi_1}{N,K,0,S}$, $\Norm{\Sigma}{N+1,K,0,S}$, $\norm{w}{N+1,K,S}
\geq \norm{w}{N,K,S}$, and $\norm{v}{N+1,K,S}$ such that
\begin{align*}
  & \norm{\Pi_1 \circ w - \Sigma_1 \circ v}{N-P-S,K-P-S}
    \notag \\
  & \qquad \leq c_1 \, \Norm{\Pi_1 - \Sigma_1}{N,K,0,P+S} + 
                c_1 \, \norm{w-v}{N-P,K-P,S} \\
  & \qquad \leq c_1 \, \Norm{\Pi - \Sigma}{N+1,K,0,P+S} + 
                c_1 \, \norm{w-v}{N-P,K-P,S} \,.
\end{align*}
To estimate the second line of the right-hand side of \eqref{e.diff1}, fix
\begin{equation}
  \label{e.rDiff}
  r = \max \{ \norm{w}{N+1,K,S}, \norm{v}{N+1,K,S} \}
\end{equation}
and set $\kV_j = \kB_r^{\kZ_j}(0)$ for $j = S, \dots, K$.  By
Lemma~\ref{l.DiffhatPiSigma} (a) below, the maps $\hat\Pi$,
$\hat\Sigma$ satisfy condition (i), i.e.,
\[
  \hat \Pi, \hat \Sigma 
  \in \kC_{N+1,K,0,S} (\{\kV_j\}, \kU, \kI; \{\kZ_j\}) \,.
\]
For fixed $\tilde{u} \in \kB$, a direct estimate verifies that
$\hat{w}((u,\tilde{u}),\mu) = \D_u w(u,\mu) \, \tilde{u}$ and
$\hat{v}((u,\tilde{u}),\mu) = \D_u v(u,\mu) \, \tilde{u}$ map
$(\kU\times\kB) \times \kI$ into each of the domains $\kV_S, \dots,
\kV_K$ of $\hat\Pi$ and $\hat\Sigma$.  Hence, $\hat{w}$ and $\hat{v}$
satisfy assumption (ii).  Then, by  the induction hypothesis, there is some
constant $c_2$ such that
\begin{align}
  & \norm{\hat \Pi \circ \hat{w} -
          \hat \Sigma \circ \hat{w}}{N-P-S,K-P-S}
    \notag \\
  & \qquad \leq c_2 \, \Norm{\hat\Pi - \hat\Sigma}{N,K,0,P+S} +
                c_2 \, \norm{\hat{w} - \hat{v}}{N-P,K-P,S}
    \notag \\
  & \qquad \leq c_2 \, \Norm{\hat\Pi - \hat\Sigma}{N,K,0,P+S} +
                c_2 \, \norm{w-v}{N+1-P,K-P,S} \,.
  \label{e.diffhatPiSigma}
\end{align}
We note that $c_2$ is a polynomial in $\tnorm{w}{N,K,S}$,
$\tnorm{v}{N,K,S}$, $\Norm{\hat\Pi}{N+1,K,0,S}$, which, by 
Lemma~\ref{l.DiffhatPiSigma} (a),   is  bounded by a  polynomial in
$\tnorm{w}{N+1,K,S}$, $\Norm{\Pi}{N+2,K,0,S}$, and $r$,  and also in $\Norm{\hat\Sigma}{N+1,K,0,S}$,
which 
is bounded  by a polynomial in
$\tnorm{v}{N+1,K,S}$, $\Norm{\Sigma}{N+2,K,0,S}$, and $r$.

To estimate the term $\Norm{\hat\Pi-\hat\Sigma}{N,K,0,P+S}$ in the
last inequality above, note that the maps
\[
  \Pi_2(w; (\hat{w},u), \mu) = \D_w \Pi(w;u,\mu) \, \hat{w}
  \quad \text{and} \quad
  \Sigma_2(w; (\hat{w},u), \mu) = \D_w \Sigma(w;u,\mu) \, \hat{w}
 \]
satisfy, for $P+S \leq \kappa \leq K$,
\[
  \Pi_2, \Sigma_2 \in \kC_{N+1,\kappa,0,S}( \{\kW_j\},
  \kV_\kappa \times \kU, \kI; \{\kZ_j\}) \,.
\]
So the induction hypothesis applies once again, asserting that there
is a constant $c_3$ such that
\[
  \norm{\Pi_2 \circ w  - \Sigma_2 \circ v}{N-P-S,\kappa-P-S} 
  \leq c_3 \, \Norm{\Pi_2-\Sigma_2}{N,\kappa,0,P+S}  
       + c_3 \, \norm{w-v}{N-P,\kappa-P,S} 
  \]
where, for all $P+S \leq \kappa \leq K$,
\[
  \Norm{\Pi_2 - \Sigma_2}{N,\kappa,0,P+S} 
  \leq r \, \Norm{\Pi-\Sigma}{N+1,K,0,P+S}
\]
and $c_3$ is a polynomial in $\Norm{\Pi}{N+2,K,0,S}$,
$\Norm{\Sigma}{N+2,K,0,S}$, $\tnorm{w}{N+1,K}$, $\tnorm{v}{N+1,K}$,
and $r$.  Therefore, by Lemma~\ref{l.technical3}, there is some
constant $c_4$, which is independent of $\Pi_2$, $\Sigma_2$, $v$, and
$w$, such that
\begin{align*}
  \Norm{\hat\Pi-\hat\Sigma}{N,K,0,P+S} 
  & = \Norm{\Pi_2 \circ w - \Sigma_2 \circ v}{N,K,0,P+S} \\
  & \leq c_4 \, \max_{P+S \leq \kappa \leq K} 
      \Norm{\Pi_2 \circ w - \Sigma_2 \circ v}{N-P-S,\kappa-P-S} \\
  & \leq r \, c_3 \, c_4 \, \Norm{\Pi-\Sigma}{N+1,K,0,P+S}
         + c_3 \, c_4 \, \norm{w-v}{N-P,K-P,S} \,.
\end{align*}
This last assertion is summarized in Lemma~\ref{l.DiffhatPiSigma}
(with $N$ replaced by $N-1$).  This concludes the inductive step in
$N$.

Second, we prove that the conclusion also holds when we increment
$K-S$ when $K<N$, holding $N$ fixed.  By Lemma~\ref{l.technical1},
\begin{multline}
  \norm{\Pi\circ w-\Sigma \circ v}{N-P-S,K-P-S+1} 
  \leq \norm{\Pi\circ w-\Sigma\circ v}{N-P-S,K-P-S+1,1} \\
  + \norm{\Pi \circ w - \Sigma \circ v}{N-P-S,0}
  + \norm{\D_\mu (\Pi\circ w- \Sigma\circ v)}{N-P-S-1,K-P-S} \,.
  \label{e.norm-splitting-2}
\end{multline}
To estimate the first term on the right-hand side, note that we can apply the
induction hypothesis on the translated scale $\tilde\kZ_j =\kZ_{j+1}$,
$\tilde\kW_j=\kW_{j+1}$.  Thus, there is a constant $c_5$ with the
required polynomial dependence such that 
\begin{align*}
  & \norm{\Pi\circ w-\Sigma\circ v}%
         {\kC_{N-P-S,K-P-S+1,1}(\kU, \kI; \{ \kZ_j \})} \\
  & \quad = \norm{\Pi\circ w-\Sigma\circ v}%
                 {\kC_{N-P-S,K-P-S}(\kU, \kI; \{ \tilde \kZ_j \})} \\
  & \quad \leq c_5 \, 
             \Norm{\Pi-\Sigma}{\kC_{N,K,0,P+S}(\{ \tilde \kW_j \},
                   \kU, \kI; \{ \tilde \kZ_j \})} 
    + c_5 \, \norm{w-v}%
                  {\kC_{N-P,K-P,S}(\kU, \kI; \{ \tilde \kZ_j \})} \\
  & \quad =  c_5 \, \Norm{\Pi-\Sigma}
                  {\kC_{N,K+1,1,P+S}(\{\kW_j\}, \kU, \kI; \{\kZ_j\})} 
    + c_5 \, \norm{w-v}{\kC_{N-P,K+1-P,1+S}(\kU, \kI; \{\kZ_j\})} \,.
\end{align*}
For the second term on the right-hand side of \eqref{e.norm-splitting-2}, we
apply the induction hypothesis on the trivial scale, obtaining that
there is a constant $c_6$ with the required polynomial dependence such
that
\[
  \norm{\Pi\circ w-\Sigma\circ v}{N-P-S,0}
  \leq c_6 \, \Norm{\Pi - \Sigma}{N,P+S,0,P+S} + 
       c_6 \, \norm{v-w}{N-P,S,S} \,.
\]
For the third term on the right-hand side of \eqref{e.norm-splitting-2}, we
estimate
\begin{multline}
  \norm{\D_\mu (\Pi \circ w) - \D_\mu(\Sigma \circ v)}{N-P-S-1,K-P-S}
  \leq \norm{\partial_\mu \Pi\circ w - 
             \partial_\mu \Sigma \circ v}{N-P-S-1,K-P-S}  \\ +
       \norm{\hat\Pi \circ  \partial_\mu w -
             \hat\Sigma \circ \partial_\mu v}{N-P-S-1,K-P-S} \,.
  \label{e.diff3}
\end{multline}
To estimate the first term on the right-hand side of \eqref{e.diff3}, notice
that $\Pi,\Sigma\in \kC_{N+1,K+1,0,S}$ implies $\partial_\mu \Pi,
\partial_\mu \Sigma \in \kC_{N+1,K+1,0,S+1}$.  Since $w,v \in
\kC_{N,K+1,1+S}$, we conclude that $\partial_\mu \Pi$, $\partial_\mu
\Sigma$, $w$ and $v$ satisfy the assumptions of the lemma.  Since
$K-S$ is not incremented, the induction hypothesis applies and proves
that there is a constant $c_7$ with the required polynomial dependence
such that
\begin{align*}
  & \norm{\partial_\mu \Pi \circ w - 
          \partial_\mu \Sigma \circ v}{N-P-(S+1),K+1-P-(S+1)}
    \notag \\
  & \qquad \leq c_7 \, 
    \Norm{\partial_\mu \Pi - \partial_\mu \Sigma}{N,K+1,0,P+S+1} 
    + c_7 \, \norm{w-v}{N-P,K+1-P,S+1}
    \notag \\
  & \qquad \leq c_7 \, 
    \Norm{\Pi - \Sigma}{N,K+1,0,P+S}
    + c_7 \, \norm{w-v}{N-P,K+1-P,S} \,.
\end{align*}
To estimate the second term on the right-hand side of \eqref{e.diff3}, 
we fix
\[
  r = \max \{ \norm{w}{N,K+1,S}, \norm{v}{N,K+1,S} \} \,,
\]
set $\kV_j = \kB_r^{\kZ_j}(0)$, and recall from above that $\hat \Pi,
\hat \Sigma \in \kC_{N,K}(\{\kV_j\}, \kU,\kI;\{\kZ_j\})$, cf.\
Lemma~\ref{l.DiffhatPiSigma}.  Then, $\partial_\mu w$ and
$\partial_\mu v$ map $\kU\times\kI$ into each of the domains $\kV_S,
\dots, \kV_K$ of $\hat\Pi$, $\hat\Sigma$.  Applying the induction
hypothesis to $\hat\Pi$, $\hat\Sigma$ and $\partial_\mu w$,
$\partial_\mu v$, we obtain that there exists a constant $c_8$ such
that
\begin{align*}
  & \norm{\hat\Pi \circ \partial_\mu w -
          \hat\Sigma \circ \partial_\mu v}{N-P-S-1,K-P-S} \\
  & \qquad\qquad 
  \leq c_8 \, \Norm{\hat\Pi-\hat\Sigma}{N-1,K,0,P+S} +
       c_8 \, \norm{\partial_\mu w - \partial_\mu v}{N-1-P,K-P,S} \\
  & \qquad\qquad 
  \leq c_8 \, \Norm{\hat\Pi-\hat\Sigma}{N-1,K,0,P+S} +
       c_8 \, \norm{w-v}{N-P,K+1-P,S} \,.
\end{align*}
The first term on the right-hand side is estimated as before, yielding
a bound of the form \eqref{e.normdiffHatPiSigma}.

We have thus found the required upper bounds for all terms on the
right-hand side of \eqref{e.norm-splitting-2}, thereby completing the inductive
step also when $K$ is incremented.
\end{proof}

In the proof of Lemma~\ref{l.diff}, we used part (a) and proved
statement (b) of the following lemma which we state for later
reference.  A proof of part (a) can be found in
\cite[Lemma~A.7]{OW2010}.

\begin{lemma} \label{l.DiffhatPiSigma}
Let $\Pi$, $\Sigma$, $w$, and $v$ be as in Lemma~\ref{l.diff}; let
$r>0$, $\kV_j = \kB_r^{\kZ_j}(0)$ for $j = S, \dots, K$, $0 \leq P
\leq \min(N-1,K)$, and
\begin{equation}
  \label{e.hatPiSigma}
  \hat\Pi(\hat{w};u,\mu) 
  = (\partial_w \Pi \circ w)(u,\mu) \, \hat{w}
  \quad \text{and} \quad 
  \hat\Sigma(\hat{w};u,\mu) 
  = (\partial_w \Sigma \circ v)(u,\mu) \,\hat{w}
\end{equation}
   The following then hold:
\begin{itemize}
\item[(a)]
$
  \hat \Pi \in \kC_{N-1,K,0,S}(\{\kV_j\}, \kU,\kI; \{ \kZ_j\})
$
with $\Norm{\hat\Pi}{N-1,K,0,S}$ bounded by a polynomial in
$\tnorm{w}{N-1,K,S}$ and $r \, \Norm{\Pi}{N,K,0,S}$, and the same
holds true for $\hat\Sigma$.

\item[(b)] There is some polynomial $c \geq 0$ in
$\Norm{\Pi}{N+1,K,0,S}$, $\Norm{\Sigma}{N+1,K,0,S}$,
$\tnorm{w}{N,K,S}$, $\tnorm{v}{N,K,S}$, and $r$ such that
\begin{equation}
  \label{e.normdiffHatPiSigma}
  \Norm{\hat\Pi-\hat\Sigma}{N-1,K,0,P+S} 
  \leq c \, \Norm{\Pi-\Sigma}{N,K,0,P+S} + 
       c \, \norm{w-v}{N-1-P,K-P,S} \,.
\end{equation}
\end{itemize}
\end{lemma}

Now we are ready to prove the result on the stability of fixed points
of contraction mappings on scales of Banach spaces.

\begin{theorem}[Stability of contraction mappings] \label{t.cm-stab}
For $N, K \in \N_0$ with $N\geq K$, let ${\kZ} = {\kZ}_0 \supset
{\kZ}_1 \supset\ldots \supset {\kZ}_{K}$ be a scale of Banach spaces,
each continuously embedded in its predecessor, let $\kW_j\subset\kZ_j$
be a nested sequence of closures of open sets, let $\kX$ be a Banach
space, and let $\kU \subset \kX$ and $\kI \subset \R$ be open.   
Assume that 
\begin{itemize}
\item[(i)] $\Pi,\Sigma \in \kC_{N+1,K}(\{\kW_j\}, \kU,\kI; \{ \kW_j\})$;
\item[(ii)] $w \mapsto \Pi(w; u, \mu)$ and $v \to \Sigma(v; u, \mu)$
are contractions on ${\kW}_j$ with contraction constants $c_j'<1$, which are
uniformly for all $u \in {\kU}$ and $\mu \in \kI$,   for  $j = 0, 1,\dots, K$.
\end{itemize}
The following then hold.
\begin{itemize}
\item[(a)] The fixed point equation $\Pi(w; u,\mu) = w$ has a unique
solution
\[
  w \in \kC_{N+1,K} (\kU, \kI; \{  {\kW}_j \})
\]
and $\tnorm{w}{N+1,K}$ is bounded by a function which is a polynomial
with non-negative coefficients in $\Norm{\Pi}{N+1,K}$ and
$(1-c_j')^{-1}$, for $j=0,1,\ldots, K$.  The same holds true for the fixed point $v=\Sigma(v;
u, \mu)$.
\item[(b)] Let $P \leq K \leq N$.  Then there is some polynomial $c$
with nonnegative coefficients in $\Norm{\Pi}{N+1,K}$,
$\Norm{\Sigma}{N+1,K}$, and $(1-c_j')^{-1}$ for $j=0,1,\ldots, K$, such
that
\[
  \norm{w-v}{N-P,K-P} \leq c \, \Norm{\Pi-\Sigma}{N,K,0,P} \,.
\]
\end{itemize}
\end{theorem}
 
Part (a) is a version of a contraction mapping theorem on a scale of
Banach spaces which was already proved in \cite[Theorem
A.9]{OW2010}.

\begin{proof}
We prove part (b) only.  It follows the same pattern as the proof of
part (a) with the additional difficulty that we need to carefully keep
track of differences in the various spaces.  As before, we use
induction in $N$ and $K$.  For $N=K=P$, we must estimate
\begin{align}
  \norm{w-v}\kZ 
  & \leq \norm{\Pi(w;u,\mu) - \Pi(v;u,\mu)}\kZ
         + \norm{\Pi(v;u,\mu) - \Sigma(v;u,\mu)}\kZ
         \notag \\
  & \leq c_0' \, \norm{w-v}\kZ 
         + \norm{\Pi(v;u,\mu) - \Sigma(v;u,\mu)}\kZ \,,
         \notag 
\end{align}
where $c_0'$ is the common contraction parameter with respect to the
$\kZ_0$ norm.  Therefore,
\[
  \norm{w-v}{\kC(\kU \times \kI; \kZ)}
  \leq \frac1{1-c_0'} \, \norm{\Pi - \Sigma}%
             {\kC (\kW_P \times \kU \times \kI; \kZ)} 
  = \frac1{1-c_0'} \, \Norm{\Pi - \Sigma}{P,P,0,P} \,. 
\]

We first prove that the conclusion also holds when we increment $N$,
holding $K$ fixed.  By Lemma~\ref{l.technical2},
\begin{equation}
  \norm{w - v}{N+1-P,K-P} 
  \leq \sup_{\|\tilde{u}\| \leq 1} \,
       \norm{(\D_u w - \D_u v) \, \tilde{u}}{N-P,K-P} + 
       \norm{w - v}{N-P,K-P} \,.
  \label{e.w-v1}
\end{equation}
By the induction hypothesis, there is a constant $c_1$ which is a
polynomial in $\Norm{\Pi}{N+1,K}$, $\Norm{\Sigma}{N+1,K}$, and
$(1-c_j')^{-1}$ for $j=0,\ldots, K$ such that
\[
  \norm{w-v}{N-P,K-P} \leq c_1 \, \Norm{\Pi - \Sigma}{N,K,0,P} \,.
\]
It remains to compute an appropriate bound for the first term on the
right-hand side of \eqref{e.w-v1}.

Note that $\tilde w (\hat u, \mu) = \D_u w(u, \mu) \, \tilde u$, where
$\hat u = (u,\tilde{u}) \in \kU \times\kB$, $\kB= \kB_1^{\kX}(0)$, is
a fixed point of the contraction map $\tilde \Pi$ given by
\begin{align}
  \tilde \Pi (\tilde w; (u,\tilde{u}), \mu)
  & = \partial_w \Pi (w(u,\mu); u, \mu) \, \tilde w
      + \partial_u \Pi (w(u,\mu); u, \mu) \, \tilde u 
    \notag \\
  & \equiv \hat \Pi (\tilde w; u, \mu) 
      + \partial_u \Pi (w(u,\mu); u, \mu) \, \tilde u \,.
  \label{e.tildePi1}
\end{align}
Similarly, $\tilde v (\hat u, \mu) = \D_u v(u, \mu) \, \tilde u$ is a
fixed point of $\tilde \Sigma$.  From part (a) we know that $w,v \in
\kC_{N+1,K}(\kU, \kI; \{\kW_j\})$.  Setting
\begin{equation}
  \label{e.r_uDeriv_stab}
  r = \max \{ \Norm{\Pi}{N+1,K}, \Norm{\Sigma}{N+1,K} \} \,
      \max_{j=0,\ldots, K} \frac1{1-c'_j}
\end{equation}
and $\kV_j = \kB_r^{\kZ_j}(0)$ for $j = 0, \dots, K$, we find by
Lemma~\ref{l.diff} (a) and Lemma~\ref{l.DiffhatPiSigma} (a) that
$\tilde\Pi, \tilde\Sigma \in \kC_{N+1,K}(\{\kV_j\},\kU \times \kB
,\kI; \{\kV_j\})$.  Hence, $\tilde\Pi$ and $\tilde\Sigma$ satisfy the
assumptions of the theorem and, by the  induction hypothesis, there is some
constant $c_2$, depending polynomially on $(1-c_j')^{-1}$ for
$j=0,\ldots, K$, $\Norm{\tilde\Pi}{N+1,K}$, and
$\Norm{\tilde\Sigma}{N+1,K}$ such that
\begin{align}
  \norm{\tilde w - \tilde v}{N-P,K-P} 
  & \leq c_2 \, \Norm{\tilde \Pi - \tilde \Sigma}{N,K,0,P} 
         \notag \\
  & \leq c_2 \, \Norm{\hat\Pi - \hat\Sigma}{N,K,0,P} 
         + c_2 \, \norm{(\partial_u \Pi \circ w - 
           \partial_u \Sigma \circ v) \, \tilde{u}}{N-P,K-P}
         \label{e.PiSigma-cm}
\end{align}
where, in the second inequality, we refer to definition
\eqref{e.hatPiSigma} of $\hat \Pi$ and $\hat \Sigma$ and to
Remark~\ref{r.no-w}.  By Lemma~\ref{l.diff} (a) and
Lemma~\ref{l.DiffhatPiSigma} (a), taking note of Remark~\ref{r.no-w},
the norms
\[ 
  \Norm{\tilde\Pi}{N+1,K} 
  \leq \Norm{\hat\Pi}{N+1,K} + 
       \norm{(\partial_u \Pi \circ w) \, \tilde{u}}{N+1,K}
\]
and $\Norm{\tilde\Sigma}{N+1,K}$ are polynomials in
$\Norm{\Pi}{N+2,K}$, $\Norm{\Sigma}{N+2,K}$, $\tnorm{w}{N+1,K}$,
$\tnorm{v}{N+1,K}$, and $r$.  Due to the definition of $r$ in
\eqref{e.r_uDeriv_stab} and part (a), these quantities, hence the
constants in \eqref{e.PiSigma-cm}, have bounds that can be chosen as
polynomials in $\Norm{\Pi}{N+2,K}$, $\Norm{\Sigma}{N+2,K}$, and
$(1-c_j')^{-1}$ for $j=0,\ldots, K$.

Applying Lemma~\ref{l.DiffhatPiSigma} to the first term on the right-hand side
of the second line of \eqref{e.PiSigma-cm} and
Lemma~\ref{l.diff} to the second term, both with $S=0$, we find that
there is a constant $c_3$ depending polynomially on
$\Norm{\Pi}{N+2,K}$, $\Norm{\Sigma}{N+2,K}$, $\norm{w}{N+1,K}$, and
$\norm{w}{N+1,K}$ such that
\begin{align*}
  \norm{\tilde w - \tilde v}{N-P,K-P} 
  & \leq c_3 \, \Norm{\Pi - \Sigma}{N+1,K,0,P} + 
         c_3 \, \norm{w-v}{N-P,K-P}
    \notag \\
  & \leq c_3 \, \Norm{\Pi - \Sigma}{N+1,K,0,P} + 
         c_4 \, \Norm{\Pi - \Sigma}{N,K,0,P}
    \notag \\
  & \leq c_5 \, \Norm{\Pi - \Sigma}{N+1,K,0,P} \,.
\end{align*}
In the second inequality we have used the induction hypothesis so that
$c_4$ and $c_5$ are polynomials in $\Norm{\Pi}{N+2,K}$,
$\Norm{\Sigma}{N+2,K}$, and $(1-c_j')^{-1}$ for $j=0,\ldots, K$.  This
concludes the inductive step in $N$.

Second, we prove that the conclusion also holds when we increment
$K<N$, holding $N$ fixed.  Recall from Lemma~\ref{l.technical1} 
that
\begin{align}  
  \norm{w-v}{N-P,K-P+1} 
  & \leq \norm{w-v}{N-P,K-P+1,1} 
    \notag \\
  & \quad + \norm{w-v}{N-P,0} 
          + \norm{\partial_\mu w - \partial_\mu v}{N-P-1,K-P} \,;
  \label{e.norm-splitting-3}
\end{align}
we will estimate the three norms on the right-hand side separately.
For the first norm note that a translation of the scale with
$\tilde\kZ_j = \kZ_{j+1}$ and the induction hypothesis show that
\begin{equation}
  \norm{w-v}{N-P,K-P+1,1}
 \leq c_6 \, \Norm{\Pi-\Sigma}{N,K+1,1,P} \,,
  \label{e.conclusion-shift}
\end{equation}
where $c_6$ is a polynomial in $\Norm{\Pi}{N+1,K+1} \geq
\Norm{\Pi}{N+1,K+1,1}$, $\Norm{\Sigma}{N+1,K+1}$, and $(1-c'_j)^{-1}$
for $j=0, \dots, K+1$.

For the second term on the right-hand side of \eqref{e.norm-splitting-3}, we
apply the induction hypothesis on the trivial scale, so that is a
constant $c_7$ such that 
\[
  \norm{w-v}{N-P,0} 
  \leq c_7 \, \Norm{\Pi - \Sigma}{N,P,0,P} \,.  
\]
For the third term on the right-hand side of \eqref{e.norm-splitting-3}, we note
that $\tilde w = \partial_\mu w$ and $\tilde v = \partial_\mu v$ are
fixed points of the respective contraction maps $\tilde \Pi$ and
$\tilde \Sigma$ of the form
\begin{align}
  \tilde \Pi (\tilde w; u, \mu)
  & = \partial_w \Pi (w(u,\mu); u, \mu) \, \tilde w
      + \partial_\mu \Pi (w(u,\mu); u, \mu) 
    \notag \\
  & \equiv \hat \Pi (\tilde w; u, \mu) 
      + \partial_\mu \Pi (w(u,\mu); u, \mu) \,.
  \label{e.tildePi2}
\end{align}
By part (a), $v,w \in \kC_{N,K+1}(\kU,\kI;\{\kZ_j\})$.  Setting
\[
  r = \max \{ \Norm{\Pi}{N,K+1}, \Norm{\Sigma}{N,K+1} \} \,
      \max_{j=0,\ldots, K} \frac1{1-c'_j}
\]
and $\kV_j = \kB_r^{\kZ_j}(0)$ for $j=0, \dots, K$, we find that, by
Lemma~\ref{l.diff} (a) and Lemma~\ref{l.DiffhatPiSigma} (a),
$\tilde\Pi, \tilde\Sigma \in \kC_{N,K}( \{\kV_j\},\kU,\kI;\{
\kV_j\})$.  Hence, $\tilde \Pi$ and $\tilde \Sigma$ satisfy the
assumptions of the theorem and, by the induction hypothesis, there is some
constant $c_8$, depending polynomially on $(1-c_j')^{-1}$,
$\Norm{\tilde\Pi}{N,K}$, and $\Norm{\tilde\Sigma}{N,K}$ such that
\begin{align}
  \norm{\tilde w - \tilde v}{N-P-1,K-P} 
  & \leq c_8 \, \Norm{\tilde \Pi - \tilde \Sigma}{N-1,K,0,P} 
         \notag \\
  & \leq c_8 \, \Norm{\hat\Pi - \hat\Sigma}{N-1,K,0,P} 
         + c_8 \, \norm{\partial_\mu \Pi \circ w - 
                        \partial_\mu \Sigma \circ v}{N-P-1,K-P}
     \label{e.PiSigma-cm2}
\end{align}
where, in the second inequality, we refer to definition
\eqref{e.hatPiSigma} of $\hat \Pi$ and $\hat \Sigma$ and to
Remark~\ref{r.no-w}.  By Lemma~\ref{l.diff} (a) and
Lemma~\ref{l.DiffhatPiSigma} (a), taking note of Remark~\ref{r.no-w},
the norms
\[ 
  \Norm{\tilde\Pi}{N,K} 
  \leq \Norm{\hat\Pi}{N,K} + 
       \norm{\partial_\mu \Pi \circ w}{N,K}
\]
and $\Norm{\tilde\Sigma}{N,K}$ are polynomials in
$\Norm{\Pi}{N+1,K+1}$, $\Norm{\Sigma}{N+1,K+1}$, $\tnorm{w}{N,K+1}$,
$\tnorm{v}{N,K+1}$, and $r$.  Due to the definition of $r$, these
quantities, and hence the constant $c_8$ in \eqref{e.PiSigma-cm2}, have a
bound that can be chosen as a polynomial in $\Norm{\Pi}{N+1,K+1}$,
$\Norm{\Sigma}{N+1,K+1}$, and $(1-c_j')^{-1}$ for $j=0,\ldots, K$.

The first term in the second line of \eqref{e.PiSigma-cm2} is
estimated by Lemma~\ref{l.DiffhatPiSigma}.  We obtain
\begin{align}
  \Norm{\hat\Pi - \hat\Sigma}{N-1,K,0,P}  
  & \leq c_9 \, \Norm{\Pi-\Sigma}{N,K,0,P}
       + c_9 \, \norm{w-v}{N-P,K-P} 
    \notag \\
  & \leq c_{10} \, \Norm{\Pi-\Sigma}{N,K,0,P} \,.
  \label{e.cmhatPi}
\end{align}
Here $c_9$ is a polynomial in $\Norm{\Pi}{N+1,K+1}\geq
\Norm{\Pi}{N,K}$, $\Norm{\Sigma}{N+1,K+1}$, $\norm{w }{N,K+1} \geq
\norm{w }{N-1,K}$ and $\norm{v}{N,K+1}$, and we have used the
induction hypothesis in the last inequality, with $c_{10}$ a
polynomial in $\Norm{\Pi}{N+1,K+1} $, $\Norm{\Sigma}{N+1,K+1}$, and
$(1-c_j')^{-1}$ for $j=0,\ldots K$.

For the second term on the right-hand side of \eqref{e.PiSigma-cm2}, note that
the hypothesis of the theorem, with $K$ replaced by $K+1$, implies
that
\[
  \partial_\mu \Pi, \partial_\mu \Sigma 
  \in \kC_{N,K+1,0,1}(\{\kW_j\},\kU,\kI;\{\kZ_j\}) \,,
\]
so that Lemma~\ref{l.diff} applied with $S=1$ yields a constant
$c_{11}$ which is a polynomial in $\Norm{\Pi}{N,K+1} \geq
\Norm{\partial_\mu\Pi}{N,K+1,0,1}$, $\Norm{\Sigma}{N,K+1}$,
$\norm{w}{N,K+1}\geq \norm{w}{N,K+1,1}$, and $\norm{v}{N,K+1}$ such
that
\begin{align}
  & \norm{(\partial_\mu \Pi)\circ w -
          (\partial_\mu \Sigma) \circ v}{N-P-1,K-P}  
    \notag   \\
  & \qquad \leq c_{11} \, \Norm{\partial_\mu \Pi - 
                             \partial_\mu \Sigma}{N,K+1,0,P+1}
              + c_{11}\, \norm{w-v}{N-P,K+1-P,1} 
    \notag \\
  & \qquad \leq c_{11} \, \Norm{\Pi - \Sigma}{N,K+1,0,P}  
    + c_{11} \, c_{12} \, \Norm{\Pi - \Sigma}{N,K+1,1,P} 
    \notag \\
  & \qquad  \leq c_{13} \, \Norm{\Pi - \Sigma}{N,K+1,0,P} \,,
  \label{e.cm_DmuPi} 
\end{align} 
where the second term in the third inequality is due to
\eqref{e.conclusion-shift}, and $c_{12}$ and $c_{13}$ depend
polynomially on the required quantities.  Inserting \eqref{e.cmhatPi}
and \eqref{e.cm_DmuPi} into \eqref{e.PiSigma-cm2} then concludes the
inductive step in $K$.
\end{proof}
 
\section*{Acknowledgments}

CW acknowledges funding by the Nuffield Foundation, by the Leverhulme
Foundation, and by EPSRC grant EP/D063906/1.  MO acknowledges support
through the ESF network Harmonic and Complex Analysis and Applications
(HCAA) and through the German Science Foundation (DFG).


\bibliographystyle{plain}

\end{document}